\documentclass [a4paper]{article}

\usepackage [latin1]{inputenc}

\usepackage{amsmath,amsfonts, amsthm, amssymb, version}
\usepackage[colorlinks=false,urlcolor=blue,linkcolor=green]{hyperref}

\numberwithin{equation}{section}

\allowdisplaybreaks[1]

\newtheorem{theorem}{Theorem}
\newtheorem{corollary}[theorem]{Corollary}
\newtheorem{proposition}[theorem]{Proposition}
\newtheorem{lemma}[theorem]{Lemma}
\newtheorem*{theorem*}{Theorem}
\newtheorem*{proposition*}{Proposition}
\newtheorem*{lemma*}{Lemma}

\theoremstyle{remark}
\newtheorem{remark}[theorem]{Remark}

\theoremstyle{definition}
\newtheorem{definition}[theorem]{Definition}
\newtheorem{example}[theorem]{Example}

\newtheorem*{question}{Question}

\newcommand{\N}{\mathbb{N}}
\newcommand{\R}{\mathbb{R}}
\newcommand{\C}{\mathbb{C}}
\newcommand{\K}{\mathbb{K}}
\newcommand{\A}{\mathfrak{A}}
\newcommand{\F}{\mathfrak{F}}

\newcommand\pnorm[2]{{\Vert #1 \Vert}_{#2}}
\newcommand\norm[1]{\Vert #1 \Vert}

\numberwithin{theorem}{section}

\begin{document}

\title{$(m,p)$-isometric and $(m,\infty)$-isometric operator tuples on normed spaces}
\author{Philipp H.W. Hoffmann and Michael Mackey}

\date{}
\maketitle
\pagestyle{myheadings}
\markright{$(m,p)$-isometric and $(m,\infty)$-isometric operator tuples on normed spaces}

\renewcommand{\thefootnote}{\fnsymbol{footnote}}
\begin{center}
\emph{Preprint of an article to appear in:\\ Asian-European Journal of Mathematics, Vol. 8, No. 2 (2015)\footnote{Copyright World Scientific Publishing Company\\ Journal URL: \url{http://www.worldscientific.com/worldscinet/aejm}}
\\$\textrm{}$\\}
\end{center}
\setcounter{footnote}{0}
\renewcommand{\thefootnote}{\arabic{footnote}}
\begin{abstract}
  We generalize the notion of $m$-isometric operator tuples on
  Hilbert spaces in a natural way to operator tuples on normed spaces. This is done by
  defining a tuple analogue of $(m,p)$-isometric operators, so-called
  $(m,p)$-isometric operator tuples. We then extend this definition
  further by introducing $(m,\infty)$-isometric operator tuples and
  study properties of and relations between these objects.
\end{abstract}

{\bf Keywords:}\ normed space, Banach space, operator tuple,
$m$-isometry,\\ $(m, p)$-isometry, $(m,\infty)$-isometry, generalised isometry\\

AMS Subject Classification: 47B99, 05A10

\section{Introduction}
Let $H$ be a Hilbert space over $\K \in \{\R, \C\}$ and let $\N$ denote the natural numbers including $0$.
If $m \in \N$, then a bounded linear operator $T \in B(H)$ is called an \emph{$m$-isometry} if, and only if, 
\begin{align}\label{original}
	\sum_{k=0}^{m}(-1)^{m-k}\binom{m}{k} T^{*k}T^{k}=0.
\end{align}
(Where, $T^0:=I$ is the identity operator for all $T \in B(X)$.) The case $m=0$ is trivial; the only example is an operator on the 
space $\{0\}$.

Originating in works of Richter \cite{ri} (the Dirichlet shift being
the standard example of a $2$-isometry) and Agler \cite{ag} in the
1980s, operators of this kind have been studied extensively by Agler
and Stankus in three papers \cite{ag-st1,ag-st2,ag-st3} and since then
attracted the interest of many other authors (see for example
\cite{bemarmart}, \cite{bemane} or \cite{Ah-He}). 

In recent years, two
generalisations of the definition of $m$-isometries have been given.
Gleason and Richter in \cite{gleari} extend the notion of $m$-isometric operators to the case of commuting $d$-tuples of bounded linear operators on a Hilbert space. The defining equation for an \emph{$m$-isometry} (or 
\emph{$m$-isometric tuple}) $T=(T_1,...,T_d) \in B(H)^d$ reads: 
\begin{align*}
	\sum_{k=0}^{m}(-1)^{k}\binom{m}{k} \sum_{|\alpha |=k}\frac{k!}{\alpha!}{T^{\alpha}}^*T^{\alpha} =0.
\end{align*}
Here, $m$ is again a non-negative integer, $\alpha$ is a multi-index, $|\alpha|$ the sum of its entries and 
$\frac{|\alpha|!}{\alpha!}=\frac{|\alpha|!}{\alpha_1!\cdots \alpha_d!}= \binom{|\alpha|}{\alpha}$) a multinomial 
coefficient.

On the other hand, the notion of $m$-isometric operators on Hilbert spaces has been generalized to operators on general 
Banach spaces in papers of Botelho \cite{bo}, Sid Ahmed \cite{sa} and Bayart \cite{ba}. In Bayart's definition, given 
$m \in \N$ and $p \in [1,\infty)$, an operator $T \in B(X)$ on a Banach space $X$ over $\K$ is called an 
\emph{$(m,p)$-isometry} if, and only if, 
\begin{align}\label{defining equation (m,p)-operator}
	\forall x \in X, \ \ \sum_{k=0}^{m}(-1)^{k}\binom{m}{k} \norm{T^{k}x}^{p}=0.
\end{align}
It is easy to see that, if $X=H$ is a Hilbert space and $p=2$, this definition coincides with the original 
definition \eqref{original} of $m$-isometries.\footnote{In the case $\K=\R$ this holds, because the operator 
$\sum_{k=0}^{m}(-1)^{m-k}\binom{m}{k} T^{*k}T^{k}$ is self-adjoint.} 
In \cite{HoMaOs} the relationship and intersection class between $(m,p)$- and $(\mu,q)$-isometries is studied. Bermúdez, Martinón and Müller 
give in \cite{bemamu} an example of an unbounded operator satisfying \eqref{defining equation (m,p)-operator}. 
We will, however, assume boundedness for convenience.

In this paper, we combine both generalisations and consider so-called $(m,p)$-isometric operator tuples on normed 
spaces, which will be defined in a natural way.

An extension of the definition of $(m,p)$-isometric operators was given in \cite{HoMaOs} to include the case 
$p = \infty$: If $m \in \N$ with $m \geq 1$, then an operator $T \in B(X)$ is called an \emph{$(m,\infty)$-isometry}
if, and only if,
\begin{align*}
	\forall x \in X, \ \ \max_{\substack{k=0,...,m \\ k \ \textrm{even}}}\|T^{k}x\| 
	= \max_{\substack{k=0,...,m \\ k \ \textrm{odd}}}\|T^{k}x\|.
\end{align*}
We will generalize this definition to the commuting tuple case in a natural way and give a conjecture on the 
intersection class of $(m,p)$-isometric and $(m,\infty)$-isometric tuples in the last part of this paper.

In the following, $X$ will denote a normed (not necessarily complete) vector space over $\K$ (unless 
stated otherwise, for example in section \ref{section spectral}). For $d \in \N$, with $d \geq 1$, let 
$T=(T_1,...,T_d) \in B(X)^d$ be a tuple of commuting bounded linear operators on $X$. (Boundedness is 
actually not essential for the definition and the basic properties of the objects we are about to discuss, but plays a 
role in the later theory.) Greek letters like $\alpha=(\alpha_1,...,\alpha_d) \in \N^d$ will denote tuples of natural 
numbers (multi-indices) or their entries, respectively. The norm or `length' of $\alpha$ will be defined by 
$|\alpha|=\sum_{j=1}^d \alpha_j$ and we set further $T^{\alpha}=T^{\alpha_1}_1\cdots T^{\alpha_d}_d$.

To denote the tuple which we obtain after removing $T_{j}$ from $T=(T_1,...,T_d)$, we will write $T'_j$  
(that is, $T'_j = (T_1,...,T_{j-1},T_{j+1},...,T_d)$ ). We use the notation $\alpha'_j$ analogously.
Finally (again if not stated otherwise), we take the exponent $p$ to be a positive real number, $p\in (0,\infty)$.

\section{Definitions and Preliminaries}

For $T \in B(X)^d$ commuting, $x \in X$ and $p \in (0,\infty)$ as above, define the sequences
$(Q^{n,p}(T,x))_{n \in \N}$ by $Q^{n,p}(T,x):=\sum_{|\alpha|=n}\frac{n!}{\alpha!}\|T^{\alpha} x\|^p$. For all $\ell \in \N$, define further the functions $P^{(p)}_{\ell}(T, \cdot):X \rightarrow \R$, by 
\begin{align*}
	P^{(p)}_{\ell}(T, x):=\sum_{k=0}^{\ell} (-1)^{\ell-k} \binom{\ell}{k} Q^{k,p}(T,x).
\end{align*}

\begin{definition}
Given $m \in \N$ and $p \in (0,\infty)$, a commuting operator tuple $T \in B(X)^d$ is called an \emph{$(m,p)$-isometry} (or \emph{$(m,p)$-isometric tuple}) if, and only if,
\begin{align*}
	P^{(p)}_m(T,x)= &\sum_{k=0}^m (-1)^{m-k} \binom{m}{k} Q^{k,p}(T,x)\\ 
	= &\sum_{k=0}^m (-1)^{m-k} \binom{m}{k} \sum_{|\alpha|=k} \frac{k!}{\alpha!}\norm{T^\alpha x}^p = 0, \ \ 
	\forall x \in X. 
\end{align*}
\end{definition}

Again, it is clear that the case $m=0$ is trivial. Further, since the operators $T_1,...,T_d$ are commuting, every permutation of an $(m,p)$-isometric tuple is also an $(m,p)$-isometric tuple.

If the context is clear, we will simply write $P_{\ell}(x)$ and $Q^n(x)$ instead of $P^{(p)}_{\ell}(T, x)$ and 
$Q^{n,p}(T,x)$.  This definition coincides with the definition of $m$-isometric tuples by Gleason and Richter if $X$ is 
a Hilbert space (and $p=2$) and has, in that context as an equivalent description, essentially already been presented in
\cite[Lemma 2.1]{gleari}.

Consequently, as one would expect, the basic theory of $(m,p)$-isometric tuples can be evolved in a similar fashion as 
in \cite{gleari}. However, we will use a different approach, based on an idea described in \cite{HoMaOs}.

Let, as in \cite[Notation 3.1]{HoMaOs}, the symbol $\mathfrak{F}$ denote the set of real functions whose domain is a 
subset of $\R$ which is invariant under the mapping $S: t\rightarrow t+1$. Further, define 
$D: \mathfrak{F} \rightarrow \mathfrak{F}$ by $Dg:=g-(g\circ S)$ for each $g\in\mathfrak F$ 
(that is, $D$ is the backward operator with difference interval $1$). 
Then 
\begin{align*}
	D^mg = \sum_{k=0}^{m}(-1)^{k}\binom{m}{k}(g\circ S^k)
\end{align*}
for all $g\in\mathfrak{F}$ and all $m \in \mathbb{N}$. Note that the set of all real sequences $\mathfrak{A}$ is a 
subset of $\mathfrak{F}$ and that
\begin{align*}
D^ma = \left(\sum_{k=0}^{m}(-1)^{k}\binom{m}{k}a_{n+k}\right)_{n \in \mathbb{N}}, \ \ \forall \ a = (a_n)_{n \in \N} 
\in \A \ \ \textrm{and} \ \ \forall \ m \in \mathbb{N}.
\end{align*}
Then $T \in B(X)^d$ is an $(m, p)$-isometric tuple, if and only if, 
\begin{align*}
\left(D^m(Q^{n}(x))_{n \in \N}\right)_0=0, \ \ \forall x \in X.
\end{align*}
Now \cite[Proposition 3.2.(ii)]{HoMaOs} states the following\footnote{This proposition is actually a
special case of a more general and well-known fact about functions defined on the natural numbers, which can, for 
example, be found in \cite[Satz 3.1]{aign}.}:

\begin{proposition}\label{prop. HoMaOs}
Let $a \in \mathfrak{A}$ and $m \in \N$. We have $D^ma=0$ if, and only if, there exists a (necessarily 
unique) polynomial function $f$ with $\deg f \leq m-1$ such that $f|_{\N}=a$. \footnote{To account for the case 
$m=0$, set $\deg 0 = - \infty$.}
\end{proposition}  

We would like to apply this fact to the sequences $(Q^{n}(x))_{n \in \N}$, to conclude that, if $T \in B(X)^d$ is 
an $(m,p)$-isometric tuple, then, for each $x \in X$, there exists a polynomial $f_x$, which interpolates 
$(Q^{n}(x))_{n \in \N}$.

Unfortunately, unlike in the situation of $(m,p)$-isometric operators (see \cite[Remark 3.6]{HoMaOs}), we can not 
immediately state that $T$ being an $(m,p)$-isometric tuple requires the whole sequence $D^m(Q^{n}(x))_{n \in \N}$ to 
be the zero-sequence. This needs some little extra work.

\begin{lemma}\label{Q^n+1}
\begin{align*}
	Q^{n+1}(x) = \sum_{j=1}^{d}Q^n(T_j x), \ \ \forall x \in X, \ \forall n \in \N.
\end{align*}
\end{lemma}
\begin{proof}
\begin{align*}
	Q^{n+1}(x) &= \sum_{|\alpha|=n+1}\frac{(n+1)!}{\alpha!}\|T^{\alpha} x\|^p
	= \sum_{|\alpha|=n+1}\frac{n!(\alpha_1 +...+\alpha_d)}{\alpha_1 ! \cdots \alpha_d !}
	\| T_1^{\alpha_1} \cdots T_d^{\alpha_d} x\|^p \\
	&= \sum_{j=1}^d \sum_{\substack{|\alpha|=n+1 \\ \alpha_j \geq 1}}
	\frac{n!\cdot \alpha_j}{\alpha_1! \cdots \alpha_d!}
	\| T_1^{\alpha_1} \cdots T_d^{\alpha_d} x\|^p \\
	&= \sum_{j=1}^d \sum_{\substack{|\alpha|=n+1 \\ \alpha_j \geq 1}}
	\frac{n! \| T_1^{\alpha_1} \cdots T_{j-1}^{\alpha_{j-1}}T_j^{\alpha_j-1}T_{j+1}^{\alpha_{j+1}} 
	\cdots T_d^{\alpha_d} T_j x\|^p}{\alpha_1 ! \cdots \alpha_{j-1}!(\alpha_j-1)!\alpha_{j+1}! \cdots \alpha_d !} 
	\\
	&= \sum_{j=1}^d \sum_{|\beta|=n}\frac{n!}{\beta!}\|T^{\beta}T_j x\|^p = \sum_{j=1}^{d}Q^n(T_j x),
	\ \ \forall x \in X, \ \forall n \in \N.
\end{align*}
\end{proof}

\begin{corollary}\label{sum D^ell}
\begin{align*}
	\left((D^\ell (Q^{n}(x))_{n \in \N}\right)_{\nu+1}= \sum_{j=1}^d \left(D^\ell (Q^{n}(T_jx))_{n \in \N}\right)_{\nu}, \ \ \forall x \in X,
\end{align*}
for all $\ell \in \N$, for all $\nu \in \N$.
\end{corollary}
\begin{proof}
\begin{align*}
	\left( D^\ell (Q^{n}(x))_{n \in \N}\right)_{\nu+1} 
	&= \sum_{k=0}^\ell (-1)^{k} \binom{\ell}{k} Q^{\nu+1+k}(x) \\
	&\overset{\textrm{Lemma} \ \ref{Q^n+1}}{=} \sum_{k=0}^\ell (-1)^{k} \binom{\ell}{k} \sum_{j=1}^{d}Q^{\nu+k}(T_j x) 
	\\
	&= \sum_{j=1}^d \left(D^\ell (Q^{n}(T_jx))_{n \in \N}\right)_{\nu}, \ \ \forall x \in X, \ \forall \ell \in \N, 
	\ \forall \nu \in \N.
\end{align*}
\end{proof}

Therefore, $(D^{m}(Q^{n}(x))_{n \in \N})_0=0$, for all $x \in X$, implies inductively\\ 
$(D^{m}(Q^{n}(x))_{n \in \N})_{\nu}=0$, for all $x \in X$, for all $\nu \in \N$. In other words:

\begin{proposition}\label{D^m=0}
$T \in B(X)^d$ is an $(m, p)$-isometry if, and only if,\\  
$D^m(Q^{n}(x))_{n \in \N}=0$, for all $x \in X$.
\end{proposition}

Before we move on, we state the following lemma, which may be of general interest. It is certainly well-known, but lacking a reference, we include the short proof.

\begin{lemma}\label{lemma for D}
Let $D: \F \rightarrow \F$ be defined as above and $(a_n)_{n \in \N}=:a \in \A$. Then
\begin{align*}
	\left( D^\ell a\right)_{\nu+1} - \left( D^\ell a\right)_{\nu} = - \left( D^{\ell +1}a\right)_{\nu}, 
	\ \ \forall \ell, \nu \in \N.
\end{align*}
\end{lemma}
\begin{proof}
By definition
\begin{align*}
	D^{\ell+1}a = D(D^{\ell}a) = D^{\ell}a - (D^{\ell}a \circ S),\ \ \forall \ell \in \N.
\end{align*}
Hence,
\begin{align*}
	\left(D^{\ell+1}a\right)_{\nu} = \left(D^{\ell}a\right)_{\nu} - \left( D^\ell a\right)_{\nu+1},\ \ \forall \ell, 
	\nu \in \N.
\end{align*}
\end{proof}

\section{Basic Properties of $(m,p)$-isometric tuples}

Our preliminary considerations allow us now to derive the basic properties of $(m,p)$-isometric 
tuples, which are analogous to those given by Gleason and Richter in \cite{gleari} in the Hilbert space case. 

Expressing Lemma \ref{lemma for D} in terms of $P_\ell(x)$ for $\nu=0$ reads:

\begin{proposition}\label{P_m+1}
\begin{align*}
	P_{\ell+1}(x)= \sum_{j=1}^d P_\ell(T_j x) - P_\ell(x), \ \ \forall x \in X, \ \forall \ell \in \N.
\end{align*}
\end{proposition}
\begin{proof}
Lemma \ref{lemma for D} gives for $\nu=0$ and $a=(Q^n(x))_{n \in \N}$, for all $\ell \in \N$,
\begin{align}\label{equation for D for n=0}
	\notag \left( D^\ell (Q^{n}(x))_{n \in \N}\right)_{1} - \left( D^\ell (Q^{n}(x))_{n \in \N}\right)_0 
	&= - \left( D^{\ell+1}(Q^{n}(x))_{n \in \N}\right)_0, \ \ \forall x \in X, \\ 
	\notag \overset{\textrm{Corollary} \ \ref{sum D^ell}}{\Leftrightarrow} \\ 
	\sum_{j=1}^d D^\ell (Q^{n}(T_jx))_{n \in \N})_0
	- \left( D^\ell (Q^{n}(x))_{n \in \N}\right)_0 
	&= - \left( D^{\ell+1}(Q^{n}(x))_{n \in \N}\right)_0, \ \ \forall x \in X.
\end{align}
By definition $\left( D^\ell (Q^{n}(x))_{n \in \N}\right)_0 = (-1)^\ell P_\ell(x)$, for all $\ell \in \N$, 
for all $x \in X$. Therefore, \eqref{equation for D for n=0} reads
\begin{align*}
	(-1)^\ell \sum_{j=1}^d P_\ell(T_j x) - (-1)^\ell P_\ell(x) = (-1)^{\ell} P_{\ell+1}(x), 
	\ \ \forall x \in X, \ \forall \ell \in \N.
\end{align*}
\end{proof}

Proposition \ref{D^m=0}, as well as Proposition \ref{P_m+1}, imply:

\begin{corollary}\label{m+1}
An $(m,p)$-isometry $T \in B(X)^d$ is an $(m+1,p)$-isometry.
\end{corollary}

Further, Proposition \ref{D^m=0} enables us to apply Proposition \ref{prop. HoMaOs} to the sequence 
$(Q^{n}(x))_{n \in \N}$, to receive the following fundamental theorem.

\begin{theorem}\label{existence of family of polynomials}
$T \in B(X)^d$ is an $(m,p)$-isometry if, and only if, there exists a (necessarily unique) family of polynomials 
$f_x:\R \rightarrow \R$, $x \in X$, of degree $\leq m-1$ with $f_x|_{\N}= (Q^{n}(x))_{n \in \N}$.
\end{theorem}

We remark that this fact has already been stated for $m$-isometric operators on Hilbert spaces by Agler and Stankus in \cite[\S 1, pages 388-389]
{ag-st1}. Further, the necessity of the existence of these polynomials has been proven by several authors: for $m$-isometric operators on Hilbert 
spaces by Bermúdez, Martinón and Negrín in \cite[Theorem 2.1]{bemane}, for $(m,p)$-isometric operators on Banach spaces
by Bayart in \cite[Proposition 2.1]{ba}, and for $m$-isometric tuples on Hilbert spaces by Gleason and Richter in \cite[Lemma 2.2 and Proposition 2.3]{gleari}.

Let now for $k, n \in \mathbb{N}$ denote the \emph{(descending) Pochhammer symbol} by $n^{(k)}$. That is,\footnote{This corrects an erroneous definition given in \cite{HoMaOs}.} 
\begin{align*}
	n^{(k)}:= \left\{ \begin{array}{ll}
					0, &\textrm{if} \ k > n,\\[1ex]
       				\binom{n}{k}k!\ , \  &\textrm{else}.
       				\end{array} \right.
\end{align*}

\begin{proposition}\label{Newtonform}
Let $m \geq 1$ and $T \in B(X)^d$ be an $(m,p)$-isometry. Then
\begin{itemize}
\item[(i)] $Q^{n}(x)=\sum\limits_{k=0}^{m-1}n^{(k)}\big(\frac{1}{k!}P_k(x)\big)$, for all $x \in X$, 
			for all $n \in \N$.
\item[(ii)] $\lim\limits_{n\rightarrow \infty} \frac{Q^{n}(x)}{n^{m-1}}=\frac{1}{(m-1)!}P_{m-1}(x) \geq 0$,
			for all $x \in X$.
\end{itemize}
\end{proposition}
\begin{proof}
(ii) follows immediately from (i).\\
For every $x \in X$, the polynomial $f_x$ interpolates the points $(n,Q^{n}(x))$. Determining the Newton form of 
$f_x$ gives (i).
\end{proof}

\begin{corollary}\label{asymptotic}
Let $m \geq 1$ and $(T_1,...,T_d) \in B(X)^d$ be an $(m,p)$-isometric tuple. Then we have
\begin{itemize}
\item[(i)]
$\left(\norm{T_j^nx}\right)_{n \in \N} \in \mathcal{O}(n^{m-1})$ for all $j \in \{1,...,d\}$, for all $x \in X$.
\item[(ii)]
$T_j^{n}(T'_j)^{\beta}x \rightarrow 0$ for $n \rightarrow \infty$, for all $\beta \in \N^{d-1}$ with 
$|\beta|\geq m$, for all $j \in \{1,...,d\}$, for all $x \in X$.
\end{itemize}
\end{corollary}
\begin{proof}
By the proposition above, for each $x \in X$, the sequence 
$\left(\frac{Q^{n}(x)}{n^{m-1}}\right)_{n\in \N}=\left(\frac{1}{n^{m-1}}\left(\sum_{|\alpha|=n}\frac{n!}{\alpha!} \norm{T^{\alpha}x}^p\right)\right)_{n\in \N}$ is a convergent sequence of sums. Since all summands are non-negative, sequences of 
summands have to be bounded. In particular, the sequences 
\begin{align*}
\left(\frac{1}{n^{m-1}}\norm{T_j^nx}^p\right)_{n \in \N}  \ \ \textrm{and} \ \  
\left(\frac{1}{n^{m-1}}\left(\frac{n!}{(n-|\beta|)!\beta!}\norm{T_j^{n-|\beta|}(T'_j)^{\beta}x}^p\right)\right)_{n \in \N}
\end{align*} 
have to be bounded for all $\beta \in \N^{d-1}$, for all $j \in \{1,...,d\}$ and all $x \in X$. This immediately gives 
(i). Noticing that
\begin{align*}
\frac{1}{n^{m-1}}\frac{n!}{(n-|\beta|)!\beta!}=\frac{n^{(|\beta|)}}{n^{m-1}\beta!} \rightarrow \infty \ \ \textrm{for} \ \ |\beta| \geq m
\end{align*}
gives (ii).
\end{proof}

\begin{proposition}
Let $m \geq 1$ and $T = (T_1,...,T_d) \in B(X)^d$ be an $(m,p)$-isometry. Then $\ker P_{m-1}$ is invariant\footnote{Note 
that by Proposition \ref{Newtonform}.(ii) and the boundedness of each $T_j$, $\ker P_{m-1}$ is indeed a closed subspace 
of $X$.} for each $T_j$ and the tuple \[T|_{\ker P_{m-1}}:= (T_1|_{\ker P_{m-1}},...,T_d|_{\ker P_{m-1}})\] 
is an $(m-1,p)$-isometry. Further, if $M \subset X$ is invariant for each $T_j$ and $T|_M$ is an $(m-1,p)$-isometry,
then $M \subset \ker P_{m-1}$.
\end{proposition}
\begin{proof}
If $T$ is an $(m,p)$-isometry, $P_m \equiv 0$. Then, by Proposition \ref{P_m+1}, 
$P_{m-1}(x)=\sum_{j=1}^d P_{m-1}(T_j x)$, for all $x \in X$. Let $x_0 \in \ker P_{m-1}$. Since $P_{m-1} \geq 0$ by
Proposition \ref{Newtonform}.(ii), $T_j x_0 \in \ker P_{m-1}$ for all $j=1,...,d$ follows. 

Note further that for every subspace $\mathcal{M}$ which is invariant for all $T_j$, we have that 
$T^{\alpha}|_{\mathcal{M}} = (T|_{\mathcal{M}})^{\alpha}$. Thus, if $x \in \ker P^{(p)}_{m-1}(T,\cdot)$, 
then $P^{(p)}_{m-1}(T|_{\ker P_{m-1}}, x) = P^{(p)}_{m-1}(T, x) = 0$. 

Similarly, if $\mathcal{M}$ is an invariant subspace for each $T_j$ such that $T|_{\mathcal{M}}$ is an 
$(m-1,p)$-isometry, then, for all $x \in \mathcal{M}$,
$0 = P^{(p)}_{m-1}(T|_{\mathcal{M}}, x) = P^{(p)}_{m-1}(T, x)$. Therefore, $x \in \ker P_{m-1}$.
\end{proof}

An $(m,p)$-isometric operator is by \cite[Proof of Theorem 3.3]{ba} an isometry on the quotient space 
$X/\ker \beta_{m-1}(T, \cdot)$ equipped with the norm $(\beta_{m-1}(T, \cdot))^{1/p}$. Here, for each $x \in X$, 
$\beta_{m-1}(T, x)= \frac{1}{(m-1)!}\sum_{j=0}^{m-1}(-1)^{m-1 -j}\binom{m-1}{j}\|T^{j}x\|^{p}$ is the leading
coefficient of the polynomial which interpolates the sequence $(\norm{T^nx}^p)_{n\in\N}$. Indeed a similar result holds 
for $(m,p)$-isometric tuples.

We will call a commuting operator tuple $T=(T_1,...,T_d)$ on a normed space $X$ an 
\emph{$\ell_p$-spherical isometry} if
\begin{align*}
	\sum_{j=1}^{d} \norm{T_jx}^p = \norm{x}^p, \ \ \forall x \in X.
\end{align*}
In the literature, $\ell_2$-spherical isometries on Hilbert spaces are referred to as just
\emph{spherical isometries}. 
Obviously $\ell_p$-spherical isometries are just $(1,p)$-isometric tuples.

The following has (in equivalent form) already been stated in \cite{ri(talk)} for $(2,2)$-isometries on Hilbert spaces.

\begin{proposition}\label{semi-norm}
Let $m \geq 1$ and $T = (T_1,...,T_d) \in B(X)^d$ be an $(m,p)$-isometry. Then $|.|_p:=(P^{(p)}_{m-1}(T, \cdot))^{1/p}$ 
is a semi-norm on $X$ with $|.|_p \leq C \norm{.}$ for some constant $C > 0$. Further, $T$ is an $\ell_p$-spherical 
isometry on the quotient space $X/\ker P^{(p)}_{m-1}(T, \cdot)$.
\end{proposition}
\begin{proof}
By Proposition \ref{Newtonform}.(ii), $|.|_p = (P^{(p)}_{m-1}(T, \cdot))^{1/p}$ is a semi-norm on $X$, 
hence a norm on $X/( \ker P^{(p)}_{m-1}(T, \cdot))^{1/p}$. That $|.|_p \leq C \norm{.}$ for some constant 
$C > 0$ follows directly from the definition of $P^{(p)}_{m-1}(T, \cdot)$ and the boundedness of $T$. Further, by 
Proposition \ref{P_m+1}, $\sum_{j=1}^{d} |T_jx|_p^p = |x|_p^p$, for all $x \in X$.
\end{proof}

\section{Examples of $(m,p)$-isometric tuples}

Non-trivial examples of $(m,p)$-isometric operator tuples are in general not easy to find. 
In the case $m=1$ this is, however, relatively simple.

\begin{example}\label{trivial example (m,p)}
Let $X$ be an arbitrary normed space and $I$ the identity operator. The pair $(\frac{1}{2}I,\frac{1}{2}I) \in B(X)^2$ is 
a $(1,1)$-isometric tuple on $X$.
\end{example}

\begin{example}
Let $T_1 = \left(\begin{array}{cc}
      \frac{1}{2} 	& 	0	\\
      0 			& 	\frac{1}{3}\\
\end{array} \right)$ and 
$T_2 = \left(\begin{array}{cc}
      \frac{\sqrt[3]{7}}{2} 	& 	0	\\
      0 			& 	\frac{\sqrt[3]{26}}{3}\\
\end{array} \right)$. Then the pair $T=(T_1, T_2)$ is a $(1,3)$-isometric tuple on $(\K^2, \pnorm{.}{3})$.
\end{example}

In \cite{ri(talk)} Richter and Sundberg state (without proof) the following sufficient condition for $(2,2)$-isometric tuples on 
finite dimensional complex Hilbert spaces:

\begin{proposition}[Richter and Sundberg \cite{ri(talk)}]
Let $z=(z_1,...,z_d) \in \C^d$ with $\pnorm{z}{2}=1$ and consider linear $V_i: \C^m \rightarrow \C^n$, 
$i \in \{1,...,d\}$, with $\sum_{i=1}^d \overline{z_i}V_i = 0$. Then the operator tuple 
$S=(S_1,...,S_d) \in B(\C^{n+m})^d$, with
\begin{align*}
	S_i = \left( \begin{array}{cc}
	z_iI_n & V_i \\
	0_m & z_iI_m \end{array}\right),
\end{align*}
is a $(2,2)$-isometric tuple.
\end{proposition}

This result leads to our next example.

\begin{example}\label{example Richter}
Let $T_1 = \left(\begin{array}{cc}
      \frac{1}{\sqrt{2}} 	& 	1	\\
      0 			& 	\frac{1}{\sqrt{2}}\\
\end{array} \right)$ and 
$T_2 = \left(\begin{array}{cc}
      \frac{1}{\sqrt{2}} 	& 	-1	\\
      0 			& 	\frac{1}{\sqrt{2}}\\
\end{array} \right)$. Then the pair $T=(T_1, T_2)$ is a $(2,2)$-isometry on $(\K^2, \pnorm{.}{2})$.
\end{example}

Further examples for $(m,p)$-isometric tuples can be easily created on the basis of $(m,p)$-isometric operators. This 
principle was used in \cite[Example 3.3, Theorem 4.1 and Theorem 4.2]{gleari}, however, since it is not stated 
explicitly there, we include it here.

\begin{proposition}\label{proposition example (m,p)-tuples made of (m,p)-operators}
Let $p \in (0, \infty)$ and $S\in B(X)$ be an $(m,p)$-isometric operator, $d \in \N$ with $d \geq 1$ and 
$z=(z_1,...,z_d) \in (\K^d, \|.\|_p)$, such that $\pnorm{z}{p}^p:=\sum_{j=1}^d|z_j|^p =1$. \footnote{The fact that 
$\norm{.}_p$ is only a quasi-norm (that is, not convex) if $0 < p < 1$ is not an issue in this case.} Then the tuple 
$T:=(z_1 S, ..., z_d S) \in B(X)^d$ is an $(m,p)$-isometric tuple.
\end{proposition}
\begin{proof}
It is clear that the operators $z_j S$ are commuting. Further, by the multinomial theorem, we have 
$(|z_1|^p + \cdots + |z_d|^p)^n = \sum_{|\alpha|=n}\frac{n!}{\alpha!}| z^\alpha |^p$. Therefore,
\begin{align*}
	Q^n(x)&= \sum_{|\alpha|=n} \frac{n!}{\alpha!}\norm{T^\alpha x}^p
	= \sum_{|\alpha|=n} \frac{n!}{\alpha!}\norm{z^\alpha S^{|\alpha|} x}^p
	= \sum_{|\alpha|=n} \frac{n!}{\alpha!} | z^\alpha |^p \norm{S^n x}^p \\
	&= \pnorm{z}{p}^{np} \norm{S^n x}^p = \norm{S^n x}^p, \ \ \forall x \in X, \ \forall n \in \N.
\end{align*}
Since $S$ is an $(m,p)$-isometric operator, 
$D^m\left(Q^n(x)\right)_{n \in \N}=D^m\left(\norm{S^nx}^p\right)_{n \in \N}=0$.
\end{proof}

(Of course, Example \ref{trivial example (m,p)} is also of this kind.)

For examples of $(m,p)$-isometric operators see for instance \cite{ri} (the Dirichlet-shift being the standard example), 
\cite{bemane} or \cite{sa}.

We now consider the special case where $T$ is an $(m,p)$-isometric tuple with one of the operators being an isometry. 

\begin{proposition}\label{isometry with nilpotent}
Let $T=(T_1,...,T_d) \in B(X)^d$ be a tuple of commuting operators and let $T_{j_0}$ be an isometry for some 
$j_0 \in \{1,...,d\}$. Then $T$ is an $(m,p)$-isometry for some $p \in (0,\infty)$ if, and only if, 
$\left(T'_{j_0}\right)^{\beta}=0$ for all $\beta \in \N^{d-1}$ with $|\beta|=m$.
In this case, $T$ is an $(m,q)$-isometry for any $q \in (0,\infty)$.
\end{proposition}
\begin{proof}
Without loss of generality, we can assume that $j_0=1$.
The necessity of $\left(T'_1\right)^{\beta}=0$ for all $\beta \in \N^{d-1}$ with $|\beta|=m$ if 
$T$ is an $(m,p)$-isometric tuple follows from Corollary \ref{asymptotic}.(ii). However, to
show equivalence, we proceed by a combinatorial approach.

Note that
\begin{align}\label{sum equal kronecker}
	\sum_{k = \ell}^m (-1)^{m-k}\binom{m}{k}\binom{k}{\ell} = \delta_{\ell, m},
\end{align}
where $\delta_{\ell, m}$ is the Kronecker-delta. To see this, write $\binom{m}{k}\binom{k}{\ell}=\frac{m!}{(m-\ell)!}\binom{m-\ell}{k-\ell}$, so that the left hand side of \eqref{sum equal kronecker} becomes
\begin{align*}
	\frac{m!}{(m-\ell)!}\sum_{k = \ell}^m (-1)^{m-\ell - (k-\ell)}\binom{m-\ell}{k-\ell}
	= \frac{m!}{(m-\ell)!}0^{m-\ell}= \delta_{\ell, m}.
\end{align*}
(The convention $0^0=1$ applies.)

Now, note further that for all $b_{k, \ell} \in \C$,
\begin{align*}
		\sum_{k=0}^{m}\sum_{\ell=0}^{k}b_{k,\ell}=\sum_{\ell=0}^{m}\sum_{k=\ell}^{m}b_{k,\ell}, 
		\ \ \forall m \in \mathbb{N}.
\end{align*}
(This can be easily seen by writing one side out and reordering the summands.)
Consequently, by combining this with (\ref{sum equal kronecker}), we get for any sequence\\ 
$(a_n)_{n \in \N} \subset \C$,
\begin{align}\label{lemma for nilpotent} 
	\sum_{k=0}^m (-1)^{m-k} \binom{m}{k} \sum_{\ell=0}^k \binom{k}{\ell}a_{\ell} 
	= \sum_{\ell=0}^m \sum_{k = \ell}^{m}(-1)^{m-k}\binom{m}{k}\binom{k}{\ell}a_{\ell} 
	= \sum_{\ell=0}^m a_{\ell} \delta_{\ell, m} = a_{m}.
\end{align}
Assume now that $T_1$ is an isometry. Then
\begin{align*}
	Q^{k,p}(T,x) &= 
	\sum_{|\alpha|=k}\frac{k!}{\alpha!}\norm{T^{\alpha}x}^p
	= \sum_{\ell=0}^{k}
	\sum_{\substack{|\alpha'_1|=\ell \\ \alpha_1=k-\ell}}\frac{k!}{\alpha!}\norm{T^{\alpha}x}^p \\
	&= \sum_{\ell=0}^{k}
	\sum_{|\alpha'_1|=\ell}\frac{k!}{(k-\ell)!(\alpha'_1)!}
	\norm{(T'_1)^{\alpha'_1}x}^p, \ \ \forall x \in X, \ \forall k \in \N.
\end{align*}
(Recall the notations $\alpha'_1:=(\alpha_2,...,\alpha_d)$ and 
$T_1'=(T_2,...,T_d)$.) Therefore 
\begin{align*}
	P^{(p)}_{m}(T,x) &= \sum\limits_{k=0}^m (-1)^{m-k} \binom{m}{k} 
	\sum\limits_{\ell=0}^k \sum_{|\alpha'_1|=\ell}
	\frac{k!\norm{(T'_1)^{\alpha'_1}x}^p}{(k-\ell)!(\alpha'_1)!}
	\\
	&=\sum\limits_{k=0}^m (-1)^{m-k} \binom{m}{k} \sum\limits_{\ell=0}^k \binom{k}{\ell}
	\sum_{|\alpha'_1|=\ell}
	\frac{\ell!\norm{(T'_1)^{\alpha'_1}x}^p}{(\alpha'_1)!}, \ \ \forall x \in X.
\end{align*}
Then by considering (\ref{lemma for nilpotent}) for the sequence
\begin{align*}
(a_n)_{n \in \N}=\left(\sum_{|\alpha'_1|=n}\frac{n!}{(\alpha'_1)!}
\norm{(T'_1)^{\alpha'_1}x}^p\right)_{n \in \N}
\end{align*}
it follows that,
\begin{align}
P^{(p)}_{m}(T,x) = 
\sum_{|\alpha'_1|=m}\frac{m!}{(\alpha'_1)!}
\norm{(T'_1)^{\alpha'_1}x}^p = 0,
\label{equation equivalence isometry with nilpotent 1} \\
\Leftrightarrow \ \
\norm{(T'_1)^{\alpha'_1}x}^p = 0,
\ \ \forall \alpha'_1 \in \N^{d-1} \ \ \textrm{with} \ \ |\alpha'_1|=m, \label{equation equivalence isometry with nilpotent 2}
\end{align}
for all $x \in X$, which is the desired equivalence.

The equivalence of \eqref{equation equivalence isometry with nilpotent 1} and 
\eqref{equation equivalence isometry with nilpotent 2} also shows
that, if $T$ is an $(m,p)$-isometry for some $p \in (0,\infty)$, it is an $(m,q)$-isometric tuple for 
any $q \in (0,\infty)$.
\end{proof}

\begin{example}\label{example isometry and nilpotent}
Let $a \in \K$ and $T_1 =  \left(\begin{array}{ccc}
      0   &   0  &  1	\\
      0  &  -1   & 	0 \\
      1  &  0  &  0\\  
\end{array} \right)$, 
$T_2 = \left(\begin{array}{ccc}
      0 	& 	0  &  0	\\
      a    &   0  &  -a \\
      0  &  0  &  0  \\
\end{array} \right)$. Then the pair $T=(T_1, T_2)$ is a $(2,p)$-isometric tuple for every $p \in (0,\infty)$ on 
$(\K^3, \pnorm{.}{q})$ for any $q \in [1,\infty]$.\footnote{Indeed, if one excludes the Hilbert norm, every isometry on $\K^n$ with respect to the $q$-norm for some $q \neq 2$, will also be an isometry with respect to any other $p$-norm. The reason for this is that the isometric matrices on 
$(\K^n,\pnorm{.}{q})$ for $q \neq 2$, $q \in [1,\infty]$ are exactly the (generalized) permutation matrices. (A proof in the real case can for example be found in \cite{LiSo}.)}
\end{example}

\section{$(m,\infty)$-isometric tuples}

Let $T=(T_1,...,T_d) \in B(X)^d$ be an $(m,p)$-isometric tuple. This is equivalent to
\begin{align*}
	\left(\sum_{\substack{k=0 \\ k \ \textrm{even}}}^m 
	\binom{m}{k} \sum_{|\alpha|=k} \frac{k!}{\alpha!}\|T^\alpha x\|^p \right)^\frac{1}{p}
	=
	\left(\sum_{\substack{k=0 \\ k \ \textrm{odd}}}^m 
	\binom{m}{k} \sum_{|\alpha|=k} \frac{k!}{\alpha!}\|T^\alpha x\|^p \right)^\frac{1}{p}, \ \ \forall x \in X.
\end{align*}
Assuming now that $T$ satisfies this for all $p \in (b, \infty)$ for some $b \geq 0$ and taking the limit for 
$p$ going to infinity, leads to the following definition.

\begin{definition}
Let $m \in \N$ with $m \geq 1$. A tuple $T=(T_1,...,T_d) \in B(X)^d$ of commuting operators is called an \emph{$(m,\infty)$-isometry} (or 
\emph{$(m,\infty)$-isometric tuple}) if, and only if,
\begin{align*}
	\max_{\substack{|\alpha|=0,...,m \\ |\alpha| \ \textrm{even}}} \norm{T^{\alpha} x} 
	= \max_{\substack{|\alpha|=0,...,m \\ |\alpha| \ \textrm{odd}}} \norm{T^{\alpha} x}, \ \ \forall x \in X.
\end{align*}
\end{definition}

(This definition extends the one appearing in \cite{HoMaOs} for operators.)

By construction, we can immediately give easy examples of these kind of tuples.

\begin{proposition}
Every $(m,p)$-isometric tuple $T=(T_1,...,T_d) \in B(X)^d$ that includes an isometry is an $(m,\infty)$-isometry.
\end{proposition}
\begin{proof}
If one of the operator $T_1,...,T_d$ is an isometry, then, by Proposition \ref{isometry with nilpotent}, $T$ is an $(m,p)$-isometric tuple for all $p \in (0,\infty)$. That $T$ is then an $(m,\infty)$-isometric tuple 
follows directly from the construction above which lead to our definition of these objects.
\end{proof}

Analogous to Proposition \ref{proposition example (m,p)-tuples made of (m,p)-operators}, one can construct further 
examples based on $(m,\infty)$-isometric operators.

\begin{proposition}\label{example made of (m,inf)-isometries}
Let $S \in B(X)$ be an $(m,\infty)$-isometric operator and\\ $z=(z_1,...,z_d) \in \K^d$, with $\pnorm{z}{\infty}=1$. 
Then the tuple $T=(z_1S,...,z_dS) \in B(X)^d$ is an $(m,\infty)$-isometry.
\end{proposition}
\begin{proof}
Since $\pnorm{z}{\infty}=1$, we have $\max_{|\alpha| = k}|z^{\alpha}| = 1$ for all $k \in \N$. Hence,
\begin{align*}
	&\max_{\substack{|\alpha|=0,...,m \\ |\alpha| \ \textrm{even}}} \norm{T^{\alpha} x} 
	= \max_{\substack{|\alpha|=0,...,m \\ |\alpha| \ \textrm{even}}} |z^{\alpha}| \norm{S^{|\alpha|} x}\\
	= &\max \{\max_{|\alpha| = k}|z^{\alpha}|\norm{S^kx}\ | \ k=0,...,m, \ k \ \textrm{even}\}\\
	= &\max_{\substack{k=0,...,m \\ k \ \textrm{even}}} \norm{S^{k} x}
	=\max_{\substack{k=0,...,m \\ k \ \textrm{odd}}} \norm{S^{k} x}
	= \max \{\max_{|\alpha| = k}|z^{\alpha}|\norm{S^kx}\ | \ k=0,...,m, \ k \ \textrm{odd}\}\\
	= &\max_{\substack{|\alpha|=0,...,m \\ |\alpha| \ \textrm{odd}}}|z^{\alpha}| \norm{S^{|\alpha|} x}
	= \max_{\substack{|\alpha|=0,...,m \\ |\alpha| \ \textrm{odd}}} \norm{T^{\alpha} x}, \ \ \forall x \in X. 
\end{align*}
\end{proof}

\begin{example}
Let $m \in \mathbb{N}$, $m \geq 1$, $p \in [1, \infty]$ and $T_p \in B(\ell_{p})$ be a weighted right-shift
operator with a weight sequence $(\lambda_{n})_{n \in \mathbb{N}}\subset \mathbb{C}$ such that
\begin{align*}
	|\lambda_{n}| \geq 1, \ \ \textrm{for} \ \ n=1,...,m-1, \ \  
	\textrm{and} \ \ |\lambda_{n}|=1, \ \ \textrm{for} \ \ n \geq m.
\end{align*}
By \cite[Example 5.11]{HoMaOs}, $T_p$ is an $(m,\infty)$-isometric operator. Then, for instance, the tuple 
$(T_p, \frac{1}{2}T_p, \frac{1}{3}T_p,...,\frac{1}{d}T_p)$ is an $(m, \infty)$-isometric tuple on $\ell_p$.
\end{example}

Since the definition of $(m,\infty)$-isometric tuples differs from the definition of $(m,\infty)$-isometric operators
basically by the replacement of the exponent $k$ by a multi-index $\alpha$, it is not surprising that its basic theory 
can be developed analogously.

First of all, we have the following (compare \cite[page 399]{HoMaOs})).

\begin{proposition}\label{proposition equivalence (m,infty)}
A commuting operator tuple $T \in B(X)^d$ is an $(m,\infty)$-isometry if, and only if, there exists an $m \in \N$, 
$m \geq 1$, with
\begin{align}\label{equivalence (m,infty)}
	\max_{\substack{|\alpha|=\ell,..,\ell+m \\ |\alpha| \ \textrm{even}}} \norm{T^{\alpha}x}
	= \max_{\substack{|\alpha|=\ell,..,\ell+m \\ |\alpha| \ \textrm{odd}}} \norm{T^{\alpha}x}, 
	\ \ \forall \ell \in \N, \forall x \in X.
\end{align}
\end{proposition}
\begin{proof}
The sufficiency of \eqref{equivalence (m,infty)} is clear. So assume now that $T \in B(X)^d$ is an 
$(m,\infty)$-isometry and let $\ell \in \N$. We only prove the case where $\ell$ is even, since the case that $\ell$ is
odd is a direct analogue. We have
\begin{align*}
	&\max_{\substack{|\alpha|=\ell,..,\ell+m \\ |\alpha| \ \textrm{even}}} \norm{T^{\alpha}x}
	= \max_{\substack{|\beta|=0,..,m \\ |\gamma| =\ell \\ |\beta| + |\gamma| \ \textrm{even}}} 
	\norm{T^{\beta}T^{\gamma}x}
	= \max_{|\gamma| = \ell} \max_{\substack{|\beta|=0,..,m \\ |\beta| \ \textrm{even}}} 
	\norm{T^{\beta}T^{\gamma}x} \\
	= &\max_{|\gamma| = \ell} \max_{\substack{|\beta|=0,..,m \\ |\beta| \ \textrm{odd}}} 
	\norm{T^{\beta}T^{\gamma}x}
	= \max_{\substack{|\beta|=0,..,m \\ |\gamma| =\ell \\ |\beta| + |\gamma| \ \textrm{odd}}} 
	\norm{T^{\beta}T^{\gamma}x}
	= \max_{\substack{|\alpha|=\ell,..,\ell+m \\ |\alpha| \ \textrm{odd}}} \norm{T^{\alpha}x}.
\end{align*}
\end{proof}

Further, using the following lemma we obtain an equivalent description of $(m,\infty)$-isometric tuples 
(compare \cite[Lemma 5.3]{HoMaOs}).

\begin{lemma}
Let $\pi(n)=n \mathop{\mathrm{mod}} 2$ denote the parity of an $n\in\N$. For any family 
$a=(a_{\alpha})_{\alpha \in \N^d} \subset \R$ and $m \in \N$, $m \geq 1$, the following are equivalent.
\begin{itemize}
	\item[(i)]
		$a$ satisfies 
		$\max\limits_{\substack{|\alpha|=\ell,...,m+\ell \\ |\alpha| \ \textrm{even}}} a_{\alpha} 
		= \max\limits_{\substack{|\alpha|=\ell,...,m+\ell \\ |\alpha| \
        \textrm{odd}}} a_{\alpha}$, \ $\forall \ell \in \N$\\[1ex]
	\item[(ii)]
 		$a$ attains a maximum and \ 
		$\max\limits_{\alpha \in \N^d} a_{\alpha} 
		= \max\limits_{\substack{|\alpha|= \ell,...,m-1+\ell \\ \pi(|\alpha|)=\pi(m-1+\ell)}} a_{\alpha}$, 
		\ $\forall \ell\in \N$.
\end{itemize}
\end{lemma}
\begin{proof}
(i)$\Rightarrow$(ii):\ We proceed by induction on $\ell$. Suppose $a=(a_{\alpha})_{\alpha \in \N^d} \subset \R$ satisfies (i) and
choose $n \in \N$ with $n \geq m$. By (i), $\max_{|\alpha|=n-m,...,n}a_{\alpha}$ is attained for at least two 
multi-indices, one of even and one of odd norm. Thus, there exists an $r < n$ and a $\beta \in \N^d$ with $|\beta|=r$, 
such that $a_{\beta} \geq a_{\alpha}$ for every $\alpha \in \N^d$ with $|\alpha|=n$. Since this holds for all 
$n \geq m$, we deduce that the family $(a_{\alpha})_{\alpha \in \N^d}$ indeed has a maximum, which is attained at an
$\alpha$ with $|\alpha| \leq m-1$. That is, $\max_{\alpha \in \N^d} a_{\alpha} = \max_{|\alpha|= 0,...,m-1} a_{\alpha}$. 
Since trivially
$\max_{|\alpha|= 0,...,m-1} a_{\alpha} \leq \max_{|\alpha|= 0,...,m} a_{\alpha}$, we actually have 
equality and by (i) can write
\begin{align*} 
\max\limits_{\alpha \in \N^d} a_{\alpha} 
= \max\limits_{\substack{|\alpha|= 0,...,m \\ \pi(|\alpha|)=\pi(m-1)}} a_{\alpha}
= \max\limits_{\substack{|\alpha|= 0,...,m-1\\ \pi(|\alpha|)=\pi(m-1)}} a_{\alpha}. 
\end{align*}
Hence, we have (ii) for $\ell_0 =0$.

Now assume that (ii) holds for some $\ell \in \N$. Then, in particular, 
$\max_{\alpha \in \N^d} a_{\alpha}= \max_{|\alpha|= \ell,...,m-1+\ell} a_{\alpha}$ and since 
$ \max_{|\alpha|= \ell,...,m-1+\ell} a_{\alpha} \leq \max_{|\alpha|= \ell,...,m+\ell} a_{\alpha}$, we again have
equality. By (i), we can omit the first $\ell$ on the right-hand side, obtaining
$\max_{\alpha \in \N^d} a_{\alpha}= \max_{|\alpha|= \ell+1,...,m+\ell} a_{\alpha}$. 
But this has to be equal to $\max_{|\alpha|= \ell+1,...,m+\ell+1} a_{\alpha}$ and again by (i), we can write
\begin{align*} 
\max\limits_{\alpha \in \N^d} a_{\alpha} 
= \max\limits_{\substack{|\alpha|= \ell+1,...,m+\ell+1 \\ \pi(|\alpha|)=\pi(m+\ell)}} a_{\alpha}
= \max\limits_{\substack{|\alpha|= \ell+1,...,m+\ell\\ \pi(|\alpha|)=\pi(m+\ell)}} a_{\alpha}. 
\end{align*}
This is (ii) for $\ell+1$.\\

(ii)$\Rightarrow$(i):\ If $a$ satisfies (ii), then, for all $\ell\in \N$,
\begin{align*}
\max_{|\alpha| \in \N^d} a_{\alpha} 
= \max_{\substack{|\alpha|= \ell,...,m-1+\ell \\ \pi(|\alpha|)=\pi(m-1+\ell)}} a_{\alpha} 
= \max_{\substack{|\alpha|= \ell,...,m+\ell \\ \pi(|\alpha|)=\pi(m-1+\ell)}} a_{\alpha}
\end{align*}
and also, by replacing $\ell$ with $\ell+1$,
\begin{align*}
\max_{|\alpha| \in \N^d} a_{\alpha} 
= \max_{\substack{|\alpha|= \ell+1,...,m+\ell \\ \pi(|\alpha|)=\pi(m+\ell)}} a_{\alpha}
\leq \max_{\substack{|\alpha|= \ell,...,m+\ell \\ \pi(|\alpha|)=\pi(m+\ell)}} a_{\alpha}
\leq \max_{\alpha \in \N^d} a_{\alpha}.
\end{align*}
This implies 
$\max\limits_{\substack{|\alpha|= \ell,...,m+\ell \\ \pi(|\alpha|)=\pi(m-1+\ell)}} a_{\alpha}
=\max\limits_{\substack{|\alpha|= \ell,...,m+\ell \\ \pi(|\alpha|)=\pi(m+\ell)}} a_{\alpha}$, 
which is (i).
\end{proof}

Combining the last two statements gives:

\begin{corollary}\label{max (m,inf)}
Let $m \in \N$ with $m \geq 1$.
A tuple of commuting operators $T=(T_1,...,T_d) \in B(X)^d$ is an $(m,\infty)$-isometry if, and only if,
for each $x \in X$ the family $(\norm{T^{\alpha}x})_{\alpha \in \N^d}$ attains a maximum and
\begin{align*}
	\max_{\alpha \in \N^d} \norm{T^{\alpha}x}
		= \max_{\substack{|\alpha|= \ell,...,m-1+\ell \\ \pi(|\alpha|)=\pi(m-1+\ell)}} \norm{T^{\alpha}x}, 
		\ \ \forall \ell\in \N.
\end{align*}
\end{corollary}

We easily deduce the following.

\begin{corollary}\label{power bounded}
Let $T=(T_1,...,T_d) \in B(X)^d$ be an $(m,\infty)$-isometry. Then $(\norm{T^{\alpha}})_{\alpha \in \N^d}$ is bounded. 
In particular, $T$ is \emph{uniformly power bounded}, that is, there exists a common $C > 0$, such that 
$\norm{T_j^n} \leq C$, for all $n \in \N$, for all $j \in \{1,...,d\}$.
\end{corollary}

Note that $C$ is, of course, given by $C = \max_{|\alpha|=0,...,m-1} \norm{T^{\alpha}}$. So we do not have to make use of the Uniform Boundedness Principle here and are not even assuming that $X$ is complete.

By simply copying the proof in the single operator case (see \cite[Proposition 6.3]{HoMaOs}), we show the following:

\begin{proposition}
Let $T \in B(X)^d$ be an $(m,\infty)$-isometric tuple. Then $T$ is an $(m+1, \infty)$-isometric tuple.
\end{proposition}
\begin{proof}
By \ref{max (m,inf)}, $\max_{\alpha \in\N^d}\|T^{\alpha}x\|$ exists
and, for all $\ell\in\N$ and $x\in X$, we have
\begin{align*}
	\max_{\alpha\in\N^d}\|T^{\alpha}x\| 
	=\max_{\substack{|\alpha|= \ell+1,...,m+\ell \\ \pi(|\alpha|) = \pi(m+\ell)}} \|T^{\alpha}x\| 
	\leq\max_{\substack{|\alpha|= \ell,...,m+\ell \\ \pi(|\alpha|)=\pi(m+\ell)}}\|T^{\alpha}x\|
	\leq\max_{\alpha\in\N^d}\|T^{\alpha}x\|
\end{align*}	
and the ensuing equality gives the result by invoking \ref{max (m,inf)} again.
\end{proof}

The case $m=1$ deserves some special attention.

We call a commuting operator tuple $T=(T_1,...,T_d) \in B(X)^d$ an 
\emph{$\ell_{\infty}$-spherical isometry} if
\begin{align*}
	\max_{j=1,...,d}\norm{T_jx} = \norm{x}, \ \ \forall x \in X.
\end{align*}
Obviously $\ell_{\infty}$-spherical isometries are just $(1,\infty)$-isometric tuples.

\begin{proposition}\label{existence of j_x for (1,inf)}
Let $T \in B(X)^d$ be an $\ell_{\infty}$-spherical isometry (i.e., a $(1, \infty)$-isometric tuple). For each $x \in X$
there exists a $j_x \in \{1,...,d\}$ such that $\norm{T^n_{j_x}x}=\norm{x}$ for all $n \in \N$.\footnote{Note that we 
make use of the continuity of the operators in the proof.}
\end{proposition}
\begin{proof}
We first show the following claim:

For each $n \in \N$ and each $x \in X$, there exists a  $j_{n, x} \in \{1,...,d\}$ with 
$\norm{T_{j_{n,x}}^kx} = \norm{x}$ for all $k \in \N$ with $k \leq n$.

Proof of the claim: 

Note first, since we have by definition $\max_{j=1,...,d}\norm{T_jx} = \norm{x}$, for all $x \in X$, that 
$\norm{T_j} \leq 1$ for all $j \in \{1,...,d\}$. Clearly then also $\norm{T^{\alpha}} \leq 1$ for all $\alpha \in \N^d$. Further, by Corollary \ref{max (m,inf)}, 
$\max_{\alpha \in \N^d}\norm{T^{\alpha}x} = \max_{|\alpha|=\ell}\norm{T^{\alpha}x} = \norm{x}$, for all $x \in X$ and 
all $\ell \in \N$ (since $m=1$). 

Therefore, for each $\ell \in \N$ and each $x \in X$, there exists an $\alpha(\ell, x)$ with $|\alpha(\ell,x)|=\ell$ and 
$\norm{T^{\alpha(\ell,x)}x}=\norm{x}$. Thus, $\norm{T^{\alpha(\ell,x)_j}x}=\norm{x}$ for all $j \in \{1,...,d\}$, as 
$\norm{T_j} \leq 1$ for all $j$ and $\norm{x}=\max_{\alpha \in \N^d}\norm{T^{\alpha}x}$. Moreover, setting 
$n = \ell d$, there exist an index $j_{n,x} \in \{1,...,d\}$ such that $\alpha_{j_{n,x}} \geq \frac{\ell}{d} = n$. \footnote{This certainly holds for any index $j_{\textrm{max}} \in \{1,...,d\}$ with
$\alpha(\ell,x)_{j_{\textrm{max}}}:=\max_{j=1,...,d}\alpha(\ell,x)_j$, which is, of course, not necessarily uniquely determined.} So $\norm{T^n_{j_{n,x}}x} = \norm{T^k_{j_{n,x}}x} = \norm{x}$ for all $k \leq n$, $k \in \N$, again as 
$\norm{T_{j_{n,x}}} \leq 1$. Thus, the claim is proved.

The rest of the proof is essentially the pigeon hole principle: For fixed $x \in X$, we
have infinitely many $n \in \N$, but only finitely many $j_{n,x} \in \{1,...,d\}$.

Fix $x \in X$ and define for each $j \in \{1,...,d\}$ the set 
\begin{align*}
A_j:=\{n \in \N \ | \ \norm{T_j^kx}=\norm{x},\ \textrm{for all} \ k \leq n,\ k \in \N\}. 
\end{align*}
By our claim, for each $n \in \N$ there is an index $j_{n,x}$, that is, every natural number resides in at least one of 
the $A_j$. Thus, since we have only finitely many sets $A_j$, at least one $A_{j_x}$ is infinite. 
$\norm{T_{j_x}} \leq 1$ then forces $A_{j_x}=\N$ as required.
\end{proof}

Therefore, we have the following remark. 

\begin{remark}\label{remark if T (m,infty)-tuple, X is union}
If $T$ is an $\ell_{\infty}$-spherical isometry, the space $X$ is the union of the closed subsets $X_{j}:=\{x \in X \ | \ \norm{x}=\norm{T_j^nx}, \ \forall n \in \N\}$. \footnote{Note that the $X_j$ are not disjoint (since $0 \in X_j$ for each $j$) and don't necessarily have trivial intersection.}
\end{remark}

In the uni-variate case, an $(m,\infty)$-isometry $T$ is an isometry under an equivalent norm, given by 
$\max_{k \in \N}\norm{T^kx}$ (\cite[Theorem 5.2]{HoMaOs}). Indeed, an analogous result holds in the tuple case.

\begin{theorem}\label{equivalent norm (m,inf)}
Let $T \in B(X)^d$ be an $(m,\infty)$-isometric tuple. Then there exists a norm $|.|_{\infty}$ on $X$ equivalent to
$\norm{.}$, under which $T$ is an $\ell_{\infty}$-spherical isometry. $|.|_{\infty}$ is given by 
$|x|_{\infty}= \max_{\alpha \in \N^d} \norm{T^{\alpha}x} = \max_{|\alpha|= 0,...,m-1} \norm{T^{\alpha}x}$, for all 
$x \in X$.
\end{theorem}
\begin{proof}
By Corollary \ref{max (m,inf)}, 
$\max_{\alpha \in \N^d} \norm{T^{\alpha}x} = \max_{|\alpha|= 0,...,m-1} \norm{T^{\alpha}x}$, for all
 
$x \in X$. Since $T_j$ is linear for each $j=1,...,d$ and the maximum preserves the triangle inequality, $|.|_{\infty}$
is indeed a norm on $X$. Further, by Corollary \ref{max (m,inf)}
\begin{align*}
\max_{j=1,...,d}\max_{\alpha \in \N^d} \norm{T^{\alpha}T_jx}
&=\max_{j=1,...,d}\max_{|\alpha|= 0,...,m-1} \norm{T^{\alpha}T_jx}\\
&=\max_{|\alpha|= 1,...,m} \norm{T^{\alpha}x}
= \max_{\alpha \in \N^d} \norm{T^{\alpha}x}, \ \ \forall x \in X, 
\end{align*}
so that $T$ is an $\ell_{\infty}$-spherical isometry with respect to $|.|_{\infty}$. Finally, we have
\begin{align*}
	\norm{x} \leq \max_{\alpha \in \N^d} \norm{T^{\alpha}x} = \max_{|\alpha|= 0,...,m-1} \norm{T^{\alpha}x}
	\leq \max_{|\alpha|= 0,...,m-1} \norm{T^{\alpha}} \cdot \norm{x}, \ \ \forall x \in X,
\end{align*}
and the two norms are equivalent.
\end{proof}

This, of course, implies immediately the next statement.

\begin{remark}
If $T$ is an $(m,\infty)$-isometric tuple, the space $X$ is the union of the closed 
subsets $X_{j,|.|}:=\{x \in X \ | \ |x|_{\infty}=|T_j^nx|_{\infty}, \ \forall n \in \N\}$.
\end{remark}

\section{Spectral Properties}\label{section spectral}

Let in this section $X$ be a complex Banach space and $m \geq 1$. As before, let $T=(T_1,...,T_d) \in B(X)^d$ 
be a tuple of commuting linear operators on $X$.\footnote{Most of the statements that we quote in this section are 
in general not true, if the operators $T_1,...,T_d$ do not commute.} A first definition of the \emph{joint spectral 
radius} of such a tuple $T$ was given by Rota and Strang in \cite{RoStra}:
\begin{align*}
	\hat{r}(T):=\lim_{k \rightarrow \infty} \max_{|\alpha|=k} \norm{T^{\alpha}}^{\frac{1}{k}}
\end{align*}
(Note that no definition of a joint spectrum is necessary for this expression to make sense.) In \cite{BeWa}, Berger and Wang give an alternative definition, which reads as follows.
\begin{align*}
	r_\ast(T):=\lim_{k \rightarrow \infty} \max_{|\alpha|=k} r(T^{\alpha})^{\frac{1}{k}},
\end{align*}
where $r(T^{\alpha}):=r(T_1^{\alpha_1}...T_d^{\alpha_d})$ is the usual spectral radius for operators.
However, in \cite[Lemma page 94]{Sol}, Soltysiak shows that we indeed have $\hat{r}(T)=r_\ast(T)$.

We further have the \emph{geometric joint spectral radius}, $r(T)$, defined as
\begin{align*}
	r(T):=\max\{\pnorm{\lambda}{2} \ | \ \lambda \in \sigma(T)\}. 
\end{align*}
Here, $\lambda=(\lambda_1,..., \lambda_d) \in \C^d$ and $\sigma(T)$ denotes the \emph{Taylor spectrum} (see \cite{Tay}).

Other kind of joint spectra include the \emph{Harte spectrum} $\sigma_H(T)$ (see \cite{Har}) and the \emph{joint (left) 
approximate point spectrum}\footnote{Harte refers to this set in \cite{Har} as \emph{left approximate point spectrum}.} 
\begin{align*}
\sigma_{\pi}(T):=\{(\lambda_1,..., \lambda_d) \in \C^d | \ \exists \ (x_k)_{k \in \N} 
\subset X \ \textrm{with} \ &\norm{x_k}=1,\ \textrm{s.th.} \\
&\lim_{k \rightarrow \infty} \sum_{j=1}^d \norm{(T_j-\lambda_j I) x_k}=0\}.
\end{align*}
All three spectra are non-void\footnote{Harte gives with \cite[Example 1.6]{Har} 
an example of a non-commuting operator pair with empty Harte spectrum.}. For the joint approximate point spectrum this 
has been shown in \cite[Theorem 1.11]{Slo-Zel}.\footnote{The definition of $\sigma_{\pi}(T)$ given in \cite{Slo-Zel} 
actually requires the existence of a net instead of a sequence, however, the proof uses a result given in \cite{Zel}, 
which is stated in terms of sequences.}

Further, it was shown in \cite{Cho-Zel} that the convex hulls of all the named spectra above coincide. Thus, the 
geometric joint spectral radius does not depend on the choice of the joint spectrum. That is, one then can replace in 
its definition the Taylor spectrum by the Harte spectrum or the joint approximate point spectrum.

Soltysiak generalizes the notion of the geometric joint spectral radius in \cite{Sol} in the following way: Define
for $p \in [1, \infty]$ the \emph{(geometric) joint $\ell_p$-spectral radius} $r_p(T)$ by
\begin{align*}
	r_p(T):=\max\{\pnorm{\lambda}{p} \ | \ \lambda \in \sigma_H(T)\}.
\end{align*}
Again, since we only consider commuting operator tuples, the $\ell_p$-spectral radius does not depend on the chosen 
spectrum.

Obviously, we have $r_2(T)=r(T)$. Further, Soltysiak shows in \cite[Theorem 2]{Sol} that 
$r_{\infty}(T)=\hat{r}(T)\ (=r_\ast(T))$. Thus, the 
$\ell_p$-spectral radii contain all variations of joint spectral radii named so far. Finally, Müller proves in 
\cite[Theorem 3]{Mue} the corresponding equalities for finite $p$:
\begin{align}\label{equation ell_p spectral radius for (m,p)-tuples p finite}
	r_p(T)= 
	\lim_{k \rightarrow \infty} \bigg( \sum_{|\alpha|=k} \frac{k!}{\alpha!}\norm{T^{\alpha}}^p \bigg)^{\frac{1}{pk}}
	=
	\lim_{k \rightarrow \infty} \bigg( \sum_{|\alpha|=k} \frac{k!}{\alpha!}r(T^{\alpha})^p \bigg)^{\frac{1}{pk}},
	\ \ p \in [1, \infty).
\end{align}
Now, Gleason and Richter prove in \cite[Proposition 3.1 and Lemma 3.2]{gleari} that the geometric spectral radius
$r(T)=r_2(T)$ of an $(m,2)$-isometric tuple on a complex Hilbert space is equal to $1$. They deliver two alternative
proofs for this, which can be easily modified to suit the case of $(m,p)$-isometric for $p\geq 1$ and $(m,\infty)$-isometric tuples on complex Banach spaces.

\begin{proposition}\label{spectral_radius}
Let $p \in [1, \infty]$ and $T \in B(X)^d$ be an $(m,p)$-isometry. Then the 
geometric joint $\ell_p$-spectral radius $r_p(T)$ of $T$ is $1$. That is,
\begin{itemize}
\item[(i)] If $p\in [1,\infty)$ and $T \in B(X)^d$ is an $(m,p)$-isometry, then 
\begin{align*}
	r_p(T)=\lim_{k \rightarrow \infty}\bigg( \sum_{|\alpha|=k} \frac{k!}{\alpha!}\norm{T^{\alpha}}^p \bigg)^{\frac{1}{pk}}=1.
\end{align*}
\item[(ii)] If $T \in B(X)^d$ is an $(m,\infty)$-isometry, then 
\begin{align*}
	\lim_{k \rightarrow \infty} \max_{|\alpha|=k}\norm{T^{\alpha}}^{\frac{1}{k}}=1.
\end{align*}
\end{itemize}
\end{proposition}
\begin{proof}
(i): By \eqref{equation ell_p spectral radius for (m,p)-tuples p finite}, 
$r_p(T)=\lim_{k \rightarrow \infty}\bigg( \sum_{|\alpha|=k} \frac{k!}{\alpha!}\norm{T^{\alpha}}^p \bigg)^{\frac{1}{pk}}$.

For all $k \in \N$, the number of summands in $\sum_{|\alpha|=k} \frac{k!}{\alpha!}\norm{T^{\alpha}}^p$ is certainly strictly less than $k^d$. Hence,
\begin{align*} 
&\lim_{k \rightarrow \infty}\bigg( \sum_{|\alpha|=k} \frac{k!}{\alpha!}\norm{T^{\alpha}}^p \bigg)^{\frac{1}{pk}}
= \lim_{k \rightarrow \infty}\max_{|\alpha|=k} \Big( \frac{k!}{\alpha!}\norm{T^{\alpha}}^p \Big)^{\frac{1}{pk}}\\ 
\textrm{and} \ \ &\lim_{k \rightarrow \infty}\sup_{\norm{x}=1}\bigg( \sum_{|\alpha|=k} \frac{k!}{\alpha!}\norm{T^{\alpha}x}^p \bigg)^{\frac{1}{pk}}
= \lim_{k \rightarrow \infty}\sup_{\norm{x}=1}\max_{|\alpha|=k} \Big( \frac{k!}{\alpha!}\norm{T^{\alpha}x}^p \Big)^{\frac{1}{pk}}.
\end{align*}
The following is taken from \cite[proof of Theorem 4]{Mue}).
\begin{align*}
	&r_p(T)= \lim_{k \rightarrow \infty}\bigg( \sum_{|\alpha|=k} \frac{k!}{\alpha!}\norm{T^{\alpha}}^p \bigg)^{\frac{1}{pk}}
	=
	\lim_{k \rightarrow \infty}\max_{|\alpha|=k} \Big( \frac{k!}{\alpha!}\norm{T^{\alpha}}^p \Big)^{\frac{1}{pk}} \\
	=
	&\lim_{k \rightarrow \infty}\max_{|\alpha|=k}\sup_{\norm{x}=1} 
	\Big( \frac{k!}{\alpha!}\norm{T^{\alpha}x}^p \Big)^{\frac{1}{pk}}
	= 
	\lim_{k \rightarrow \infty}\sup_{\norm{x}=1} \max_{|\alpha|=k} 
	\Big( \frac{k!}{\alpha!}\norm{T^{\alpha}x}^p \Big)^{\frac{1}{pk}} \\
	=
	&\lim_{k \rightarrow \infty}\sup_{\norm{x}=1}
	\bigg( \sum_{|\alpha|=k} \frac{k!}{\alpha!}\norm{T^{\alpha}x}^p \bigg)^{\frac{1}{pk}}
	= \lim_{k \rightarrow \infty} \sup_{\norm{x}=1}(Q^{k,p}(T, x))^{\frac{1}{pk}}.
\end{align*}
Since $\sigma_{H}(T) \neq \emptyset$ is compact, 
$r_p(T) < \infty$ exists finitely. Hence, the equation above shows that the function 
$\left(Q^{k,p}(T,\cdot)\right)^{\frac{1}{pk}}$, restricted to the closed unit ball $\overline{B(0;1)}$ of $X$, converge 
uniformly to $r_p(T)$. Thus, they converge point-wise.

The remaining parts of the proof are now almost identical to the proof of \cite[Proposition 3.1]{gleari}. By Proposition \ref{Newtonform}.(ii),
\begin{align*}
	\lim_{k \rightarrow \infty}\frac{Q^{k,p}(T,x)}{k^{m-1}}
	=\frac{1}{(m-1)!}P^{(p)}_{m-1}(T,x) \geq 0, \ \ \forall x \in X.
\end{align*}
Assuming that $m$ is the smallest natural number, for which $T$ is $(m,p)$-isometric, prompts that there exist vectors
$x \in \overline{B(0;1)}$ for which the inequality on the right is strict.
Hence, since $\lim_{k \rightarrow \infty} \big(k^{m-1}\big)^\frac{1}{pk}=1$, we have for all $x \in \overline{B(0;1)}$ 
with $P^{(p)}_{m-1}(T,x) \neq 0$,
\begin{align*}
	&r_p(T)
	=\lim_{k \rightarrow \infty} (Q^{k,p}(T,x))^{\frac{1}{pk}}
	= \lim_{k \rightarrow \infty} \bigg(\frac{(Q^{k,p}(T,x)}{k^{m-1}}\bigg)^{\frac{1}{pk}} \\
	= &\lim_{k \rightarrow \infty}\bigg(\frac{1}{(m-1)!}P^{(p)}_{m-1}(T,x)\bigg)^{\frac{1}{pk}}=1.
\end{align*}
(ii): By  Corollary \ref{power bounded} the sequence
$(\max_{|\alpha|=n}\norm{T^{\alpha}})_{\alpha \in \N^d}$ is bounded. The statement 
follows if we show that this sequence is also bounded below. 

By Proposition \ref{equivalent norm (m,inf)}, we have 
$C \cdot \norm{T^{\alpha}x} \geq |T^{\alpha}x|_{\infty}$ for all $x \in X$ and 
$\alpha \in \N^d$ (with
$C = \max_{|\alpha|= 0,...,m-1} \norm{T^{\alpha}}$). This implies
\begin{align*} 
C \cdot \max_{|\alpha|=n} \norm{T^{\alpha}x} \geq \max_{|\alpha|=n}|T^{\alpha}x|_{\infty} = |x|_{\infty} \geq \norm{x}, \ \ \forall x \in X, \ \forall n \in \N.
\end{align*}
Here, the equality is due to Corollary \ref{max (m,inf)}, since $T$ is an $\ell_{\infty}$-spherical isometry 
w.r.t. $|.|_{\infty}$.

In particular, we have $C \cdot \max_{|\alpha|=n} \norm{T^{\alpha}x} \geq \norm{x}$ for all $x \in X$, and then 
\begin{align*}
\sup_{\norm{x}=1}\left( C \cdot \max_{|\alpha|=n} \norm{T^{\alpha}x}\right) 
=C \cdot \max_{|\alpha|=n} \sup_{\norm{x}=1} \norm{T^{\alpha}x}
= C \cdot \max_{|\alpha|=n} \norm{T^{\alpha}} \geq 1, \ \ \forall n \in \N.
\end{align*}
\end{proof}

In \cite[Lemma 3.2]{gleari} it is shown that, if $T$ is an $(m,2)$-isometry on a complex Hilbert space, then
$\pnorm{\lambda}{2}=1$, for all $\lambda \in \sigma_{\pi}(T)$. That is, its joint
approximate point spectrum lies in the boundary of the $d$-dimensional unit sphere and, therefore, 
one gets again that the geometric joint spectral radius of an $(m,2)$-isometry is equal to $1$. 
To obtain this result, Gleason and Richter show (see \cite[page 187]{gleari}) that $\lambda \in \sigma_{\pi}(T)$ if, 
and only if, there exists a sequence $(x_n)_{n \in \N} \subset X$, with $\norm{x_n}=1$, for all $n \in \N$, such that 
$(T^{\alpha}-\lambda^{\alpha})x_n \rightarrow 0$ $(n \rightarrow \infty)$, for all $\alpha \in \N^d$. Using this 
fact, we obtain a generalisation of \cite[Lemma 3.2]{gleari}:

\begin{proposition}\label{approximate point spectrum}
Let $p \in [1, \infty]$. Then the joint approximate point spectrum $\sigma_{\pi}(T)$ of 
an $(m,p)$-isometric tuple $T \in B(X)^d$ is a subset of the $d$-dimensional complex unit sphere with respect to the 
$p$-norm.
\end{proposition}
\begin{proof}
Let $\lambda \in \sigma_{\pi}(T)$. 

If $p \in [1, \infty)$ and $T \in B(X)^d$ is an $(m,p)$-isometric tuple, then 
there exists a sequence $(x_n)_{n \in \N} \subset X$, with $\norm{x_n}=1$, for
all $n \in \N$, such that
\begin{align*}
	0&=\lim_{n \rightarrow \infty }\sum_{k=0}^m (-1)^{m-k} \binom{m}{k} \sum_{|\alpha|=k} \frac{k!}{\alpha!}
	\norm{T^\alpha x_n}^p \\
	&= \sum_{k=0}^m (-1)^{m-k} \binom{m}{k} \sum_{|\alpha|=k} \frac{k!}{\alpha!}
	|\lambda^{\alpha}|^p = (1 - \pnorm{\lambda}{p}^p)^m.
\end{align*} 
$\pnorm{\lambda}{p}=1$, for all $\lambda \in \sigma_{\pi}(T)$, follows immediately.

If $T \in B(X)^d$ is an $(m,\infty)$-isometric tuple, there exists a sequence $(x_n)_{n \in \N} \subset X$, with $\norm{x_n}=1$, for all $n \in \N$, such that
\begin{align*}
		&\lim_{n \rightarrow \infty}
		\max_{\substack{|\alpha|=0,...,m \\ |\alpha| \ \textrm{even}}}\|T^{\alpha}x_{n}\| =	
		\max_{\substack{|\alpha|=0,...,m \\ |\alpha| \ \textrm{even}}}|\lambda^{\alpha}| \\ 
		\textrm{and} \ \ \  
		&\lim_{n \rightarrow \infty} 
		\max_{\substack{|\alpha|=0,...,m \\ |\alpha| \ \textrm{odd}}}\|T^{\alpha}x_{n}\| =	
		\max_{\substack{|\alpha|=0,...,m \\ |\alpha| \ \textrm{odd}}}|\lambda^{\alpha}|.
	\end{align*}
	Since $T$ is an $(m,\infty)$-isometry and by uniqueness of limits, we therefore have
	\begin{align*}
		\max_{\substack{|\alpha|=0,...,m \\ |\alpha| \ \textrm{even}}}|\lambda^{\alpha}|
		&= \max_{\substack{|\alpha|=0,...,m \\ |\alpha| \ \textrm{odd}}}|\lambda^{\alpha}| \\
		\Leftrightarrow \ \ \
		\max_{\substack{|\alpha|=0,...,m \\ |\alpha| \ \textrm{even}}}\pnorm{\lambda}{\infty}^{|\alpha|}
		&= \max_{\substack{|\alpha|=0,...,m \\ |\alpha| \ \textrm{odd}}}\pnorm{\lambda}{\infty}^{|\alpha|}.
\end{align*}
The fact that we are equating an even power of $\pnorm{\lambda}{\infty}$ with an odd power forces 
$\pnorm{\lambda}{\infty}$ to be $0$ or $1$. However, if $\lambda=0$, then the maximum on the left hand side is $1$, 
reached at $|\alpha|=0$, and the right hand side is $0$. Therefore, we must have $\pnorm{\lambda}{\infty}=1$. 
\footnote{Note also that, by the proof of Proposition \ref{spectral_radius}.(ii), we have $C \cdot \max_{|\alpha|=n} \norm{T^{\alpha}x} \geq \norm{x}$ for all $x \in X$ 
and all $n \in \N$, with $C > 0$. Thus, $\left(\max_{|\alpha|=0,...,m}\|T^{\alpha}x_{n}\|\right)_{n \in \N}$ is bounded below, which also shows that $\lambda = 0$ cannot occur.}
\end{proof}

\section{On the intersection class of $(m,p)$- and $(m,\infty)$-isometric tuples}

It is known (see \cite[Proposition 6.1]{HoMaOs}) and easy to see, that an $(m,p)$-isometric operator is simultaneously
an $(m,\infty)$-isometric operator if, and only if, it is an isometry. 

A natural analogue of this statement would appear to be ``an $(m,p)$-isometric tuple is simultaneously an 
$(m,\infty)$-isometric tuple if, and only if, it is an $\ell_p$-spherical isometry (or an $\ell_{\infty}$-spherical isometry)''. However, such a statement cannot be true. 

\begin{example}
Let $T_1 \in B(X)$ be an isometric operator and $T=(T_1,...,T_d) \in B(X)^d$ an $(m,p)$-isometry. By Proposition 
\ref{isometry with nilpotent}, $T$ is an $(m,p)$-isometry for every $p \in (0,\infty)$ and, hence, by definition an 
$(m,\infty)$-isometry. However, in general $T$ does not need to be an $\ell_{p}$-spherical or an 
$\ell_{\infty}$-spherical isometry, as Example \ref{example isometry and nilpotent} shows.
\end{example}

It is currently unknown what the intersection of the set of all $(m,p)$-isometric and all 
$(m,\infty)$-isometric tuples on a given normed space $X$ actually is. Looking at the joint approximate point 
spectrum (in the complex Banach space case if $p \geq 1$) gives some information.

\begin{remark}
Let $X$ be a complex Banach space and $p \in [1,\infty)$. Let further $T=(T_1,...,T_d) \in B(X)^d$ be an 
$(m,p)$-isometric and a $(\mu,\infty)$-isometric tuple. Then every 
$\lambda \in \sigma_{\textrm{ap}}(T)$ satisfies $\pnorm{\lambda}{p}=\pnorm{\lambda}{\infty} =1$ by Proposition 
\ref{approximate point spectrum}. Consequently, since 
$\sigma_{\textrm{ap}}(T) \subset \sigma_{\textrm{ap}}(T_1) \times \cdots \times \sigma_{\textrm{ap}}(T_d)$, one operator $T_{j_0}$ has spectral 
radius $r(T_{j_0}) \geq 1$ and the remaining operators $T_i$, $i \neq j_0$, are not bounded below and in particular not 
invertible.
\end{remark} 

More specific results can so far only be given in special cases.

The case where our tuple is constructed by using an $(m,\infty)$-isometric operator (i.e. by applying 
Proposition \ref{example made of (m,inf)-isometries}) is easy and we consider it first.

\begin{proposition}
Let $T=(T_1,...,T_d) \in B(X)^d$ be an $(\mu,\infty)$-isometric tuple of the form of 
Proposition \ref{example made of (m,inf)-isometries}. That is, $T_j=z_jS$, where $S \in B(X)$ is an 
$(\mu,\infty)$-isometric operator and $z:=(z_1,...,z_d) \in \K^d$ with $\norm{z}_{\infty}=1$. Assume further that 
$T$ is additionally an $(m,p)$-isometric tuple. Then the operator $S$ is an isometry, 
$T=(0,...,0,z_{j_0}S,0,...,0)$ with $|z_{j_0}|=1$ for some $j_0 \in \{1,...,d\}$ and $z_{j_0}S$ is (trivially) also an 
isometry. 
\end{proposition}
\begin{proof}
Since $T$ is an $(m,p)$-isometry, we have for all $x \in X$,
\begin{align*}
	Q^n(x)= \sum_{|\alpha|=n} \frac{n!}{\alpha!}\| T^\alpha x\|^p 
	= \sum_{|\alpha|=n} \frac{n!}{\alpha!}\| z^\alpha S^{|\alpha|} x\|^p
	= \pnorm{z}{p}^{np} \| S^n x\|^p,
\end{align*}
by the multinomial theorem. Then $D^m\left(Q^n(x)\right)_{n \in \N}=0$ implies that $\pnorm{z}{p}S$ is an 
$(m,p)$-isometric operator and, consequently, the sequence $\left(\pnorm{z}{p}^{np}\norm{S^nx}^p\right)_{n \in \N}$ is 
a polynomial of degree $\leq m-1$ for all $x \in X$.

Since $\pnorm{z}{\infty}=1$, we have $\pnorm{z}{p} \geq 1$. If this inequality was strict, 
the sequence $\left(\pnorm{z}{p}^{np}\right)_{n \in \N}$ would 
grow exponentially. But since $S$ is a 
$(\mu,\infty)$-isometric operator, we have for all $\ell \in \N$, for all $x \in X$,
$\max_{n \in \N}\norm{S^nx}=\max_{\ell,...,\ell+\mu-1}\norm{S^nx}$. This would contradict the polynomial growth of 
$\left(\pnorm{z}{p}^{np}\norm{Sx}^p\right)_{n \in \N}$.

Therefore, we have $\pnorm{z}{\infty}=\pnorm{z}{p}=1$, which gives $|z_{j_0}|=1$ for some 
$j_0 \in \{1,...,d\}$ and $z_i=0$ for all $i \neq j_0$. In particular, 
$S$ has to be an $(m,p)$-isometric operator, which forces
$S$ by \cite[Proposition 6.1]{HoMaOs}, and therefore $z_{j_0}S$, to 
be an isometry.
\end{proof}

To prove further results, we first state a series of lemmata.

\begin{lemma}\label{when (1,inf) implies isometry}
Let $T=(T_1,...,T_d) \in B(X)^d$ be an $\ell_\infty$-spherical isometry (i.e., a $(1,\infty)$-isometric tuple). 
For $j \in \{1,...,d\}$ let, as in Remark \ref{remark if T (m,infty)-tuple, X is union},\\ 
$X_j=\{x \in X \ | \ \norm{x}=\norm{T_j^nx}, \ \forall n \in \N\}$. 

Then one operator $T_{j_0}$ is an isometry and all 
other operators $T_i$ with $i \neq j_0$ are nilpotent if, and only if,
there exists $\nu \in \N$ such that, $X_j \subset \ker T^{\nu}_i$, 
for all $i \neq j$, $i, j \in \{1,...,d\}$.
\end{lemma}
\begin{proof}
``$\Rightarrow$'': If one operator $T_{j_0}$ is an isometry and all 
other operators $T_i$ with $i \neq j_0$ are nilpotent, we have $X_{j_0}=X$ and $X_i = \{0\}$ for all $i \neq j_0$. Setting $\nu:=\max\{n \in \N \ | \ T_i^n=0 , \ i = 1,..., d, i \neq j_0 \}$ gives
$\ker T^{\nu}_i = X$ for all $i \neq j_0$. Since $\ker T^{\nu}_{j_0} = \{0\}$, the statement follows.

``$\Leftarrow$'': By Remark \ref{remark if T (m,infty)-tuple, X is union}, $X=\bigcup_{j=1,...,d} X_j$. 
Since by assumption each $X_j \subset \ker T_i^{\nu}$ 
for all $i \in \{1,...,d\}$ with $i \neq j$, we have 
\begin{align*}
	X = \bigcup_{j=1,...,d} X_j \subset \bigcup_{j=1,...,d} \bigcap_{\substack{i=1,...,d \\ i \neq j}}\ker T^{\nu}_i.
\end{align*}
This forces $\bigcap_{\substack{i=1,...,d \\ i \neq j_0}}\ker T^{\nu}_i = X$ for some $j_0 \in \{1,...,d\}$ (since each intersection is a linear space). 
Hence, $\ker T^{\nu}_i = X$ for all $i \neq j_0$, which means $T^{\nu}_i = 0$ and, thus, 
$X_i = \{0\}$ for all $i \neq j_0$. Then we must have $X_{j_0} = X$ and $T_{j_0}$ is an isometry.
\end{proof}

\begin{lemma}\label{T_i^mT_j^m=0}
Let $T =(T_1,...,T_d) \in B(X)^d$ be an $(m,p)$-isometric tuple and also a $(\mu, \infty)$-isometric tuple. Then
for all $\gamma=(\gamma_1,...,\gamma_d) \in \N^d$ with the property 
$|\gamma'_j| \geq m$ for every
$j \in \{1,...,d\}$, we have $T^{\gamma}=0$. In particular, $T_i^mT_j^m=0$ for every $i \neq j$, $i,j \in \{1,...,d\}$.
\end{lemma}
\begin{proof}
Lets first consider the case $\mu=1$. If $T$ is a $(1,\infty)$-isometric
tuple, by Corollary \ref{max (m,inf)},
\begin{align*}
	\max_{\alpha \in \N^d} \norm{T^{\alpha}x} = \norm{x} = \max_{|\alpha|=\ell} \norm{T^{\alpha}x}, 
	\ \forall \ell \in \N, \ \forall x \in X.
\end{align*}
So $\left(\max_{|\alpha|=\ell} \norm{T^{\alpha}x}\right)_{\ell \in \N}$ is a constant sequence for all $x \in X$. In 
particular, $\left(\max_{|\alpha|=\ell} \norm{T^{\alpha}T^{\gamma}x}\right)_{\ell \in \N}$ is constant for any 
multi-index $\gamma \in \N^d$, for all $x \in X$.

Since $T$ is an $(m,p)$-isometric tuple, for any $x \in X$, any $\beta \in \N^{d-1}$ with $|\beta| \geq m$ and any 
$j \in \{1,...,d\}$, by Corollary \ref{asymptotic}.(ii), $T_j^n(T'_j)^{\beta}x \rightarrow 0$ 
for $n \rightarrow \infty$. We will show that this implies, given a $\gamma \in \N^{d}$ with the property 
$|\gamma'_j| \geq m$ for every $j \in \{1,...,d\}$, that $\norm{T^{\alpha}T^{\gamma}x} \rightarrow 0$ as 
$|\alpha| \rightarrow \infty$, for all $x \in X$.

So take a $\gamma \in \N^{d}$ with the property $|\gamma'_j| \geq m$ for
every $j \in \{1,...,d\}$. Then for any $x \in X$, any $j \in \{1,...,d\}$ and for all $\varepsilon > 0$, there exists 
an $N_{\varepsilon}(x,j) \in \N$ such that $\norm{T_j^n(T'_j)^{\gamma'_j}x} \leq \varepsilon$, 
for all $n \geq N_{\varepsilon}(x,j)$, by Corollary \ref{asymptotic}.(ii).

But since we have only finitely many $j$, by simply taking the maximum $N_{\varepsilon}(x)$ of all 
$N_{\varepsilon}(x,j)$, we get that for any $x \in X$, for all $\varepsilon > 0$, 
$\norm{T_j^n(T'_j)^{\gamma'_j}x} \leq \varepsilon$, for all $j \in \{1,...,d\}$, for all $n \geq N_{\varepsilon}(x)$.

For all $\alpha \in \N^d$ with $|\alpha|=\ell$, we have 
$\alpha_{j_{\textrm{max}}}:=\max_{j=1,...,d}\alpha_j \geq \frac{\ell}{d}$, for all $\ell \in \N$. But then, for any chosen $x \in X$ and for all $ \varepsilon > 0$, 
there exists an $M_{\varepsilon}(x) \in \N$ such that
$\norm{T_{j_{\textrm{max}}}^{\alpha_{j_{\textrm{max}}}+\gamma_{j_{\textrm{max}}}}
(T'_{j_{\textrm{max}}})^{\gamma'_{j_{\textrm{max}}}}x} \leq \varepsilon$ 
for all $\alpha \in \N^d$
with $|\alpha|=\ell$, for all $\ell \geq M_{\varepsilon}(x)$. \footnote{Of course, the index $j_{\textrm{max}}$ is not 
uniquely determined and may also be different for every $\alpha$.} 

Therefore, 
\begin{align*}
	\norm{T^{\alpha}T^{\gamma}x} &= \norm{(T'_{j_{\textrm{max}}})^{\alpha'_{j_{\textrm{max}}}}
	T_{j_{\textrm{max}}}^{\alpha_{j_{\textrm{max}}}+\gamma_{j_{\textrm{max}}}}
	(T'_{j_{\textrm{max}}})^{\gamma'_{j_{\textrm{max}}}}x} \\
	&\leq \norm{(T'_{j_{\textrm{max}}})^{\alpha'_{j_{\textrm{max}}}}} \cdot 
	\norm{T_{j_{\textrm{max}}}^{\alpha_{j_{\textrm{max}}}+\gamma_{j_{\textrm{max}}}}
	(T'_{j_{\textrm{max}}})^{\gamma'_{j_{\textrm{max}}}}x}  \\
	&\leq \norm{(T'_{j_{\textrm{max}}})^{\alpha'_{j_{\textrm{max}}}}} \cdot \varepsilon, \ \ 
	\forall \alpha \in \N^d \ \ \textrm{with} \ \ |\alpha|=\ell, \ \forall \ \ell \geq M_{\varepsilon}(x).
\end{align*}
Now, since $T$ is a $(1,\infty)$-isometric tuple, $\norm{T_j} \leq 1$, for all $j \in \{1,...,d\}$.
Thus, $\norm{T^{\alpha}T^{\gamma}x} \leq \varepsilon$, for all $\alpha \in \N^d$ with $|\alpha|=\ell$, 
for all $\ell \geq M_{\varepsilon}(x)$.

Then
\begin{align*}
	\norm{T^{\gamma}x}= \max_{|\alpha|=\ell} \norm{T^{\alpha}T^{\gamma}x} \rightarrow 0, \ \textrm{as} \ \ell 
	\rightarrow \infty.
\end{align*}
Since $x$ was chosen arbitrarily, $T^{\gamma}=0$ follows.

Now consider the case $\mu > 1$ and let $T$ be a $(\mu, \infty)$-isometry. By Theorem \ref{equivalent norm (m,inf)}, $T$
is a $(1,\infty)$-isometric tuple with respect to the norm $|.|_{\infty}$ on $X$, where $|.|_{\infty}$ is
equivalent to $\norm{.}$. Hence, for any $x \in X$, any $\beta \in \N^{d-1}$ with $|\beta| \geq m$ and any 
$j \in \{1,...,d\}$, $T_j^n(T'_j)^{\beta}x$ converges to $0$ for $n \rightarrow \infty$ under 
$|.|_{\infty}$. By repeating the argument from above \footnote{Note that we do not have to assume that $T$ is an 
$(m,p)$-isometry w.r.t. $|.|_{\infty}$}, we then get that
\begin{align*}
	|T^{\gamma}x|_{\infty}= \max_{|\alpha|=\ell} |T^{\alpha}T^{\gamma}x|_{\infty} \rightarrow 0, \ \textrm{as} \ \ell 
	\rightarrow \infty.
\end{align*}
Again, $T^{\gamma}=0$ follows.
\end{proof}

\begin{corollary}\label{Cor T_i^mT_j^m=0}
Let $T =(T_1,...,T_d) \in B(X)^d$ be an $(m,p)$-isometric tuple for some $m \geq 1$ and also a $(\mu, \infty)$-isometric 
tuple. If $T^{\alpha} \neq 0$ with $|\alpha|=n$, then $\alpha$ is a permutation of 
$(n-|\beta|,\beta_1,...,\beta_{d-1})$, where $|\beta|\leq m-1$. I.e., 
$T^{\alpha}=T_j^{n-|\beta|}(T'_j)^{\beta}$ for some $j \in \{1,...,d\}$ and some $\beta \in \N^{d-1}$  with $|\beta|\leq m-1$. \footnote{Of course, we actually have $|\beta|\leq \min\{n, m-1\}$. The main point is, however, that $|\beta|\leq m-1$.}
\end{corollary}
\begin{proof}
If $d=1$ there is nothing to show, so assume $d \geq 1$.

Let $|\alpha|=n$ and chose a $j \in \{1,...,d\}$ with $\alpha_j \neq 0$. Then we can write $\alpha_j=n-|\beta|$ and 
$T^{\alpha}=T_j^{n-|\beta|}(T'_j)^{\beta}$ for $\beta=\alpha'_j \in \N^{d-1}$. We have to show that, if $T^{\alpha} \neq 0$, then 
$|\beta|\leq m-1$, or we can reorder and write $T^{\alpha}=T_k^{n-|\tilde{\beta}|}(T'_k)^{\tilde{\beta}}$ 
for some $k \in \{1,...,d\}$ and some $\tilde{\beta} \in \N^{d-1}$ with $|\tilde{\beta}| \leq m-1$.

Since $n-|\beta| = \alpha_j \geq 1$, the statement holds trivially if $n \leq m$. So assume $n \geq m+1$.

Certainly, by Lemma \ref{T_i^mT_j^m=0}, if $n-|\beta| \geq m$ and $|\beta| \geq m$, 
$T_j^{n-|\beta|}(T'_j)^{\beta}=0$. This means, if $T^{\alpha} \neq 0$, we have we must have 
$|\beta| \leq m-1$ or $|\beta| \geq n-m+1$.

If we have $|\beta| \leq m-1$ we are done, so assume $|\beta|\geq m$ and $|\beta| \geq n-m+1$.

Now, if the biggest entry of $\beta$, 
$\beta_{j_{\textrm{max}}}:=\max_{j=1,...,d-1}\beta_j$ (where again, the index $j_{\textrm{max}}$ is not 
necessarily unique) satisfies 
$\beta_{j_{\textrm{max}}} \geq n-m+1$, then we have\\ 
$- \beta_{j_{\textrm{max}}} + n =  |\beta'_{j_{\textrm{max}}}| + n - |\beta| \leq m-1$ and we 
let $\tilde{\beta}$ be a multi-index consisting of the entries of $\beta'_{j_{\textrm{max}}}$ and the entry 
$n - |\beta|$ w.r.t. some permutation. 
Then $T^{\alpha}=T_{j_{\textrm{max}}}^{\beta_{j_{\textrm{max}}}}
\left(T'_{j_{\textrm{max}}}\right)^{\tilde{\beta}} = 
T_{j_{\textrm{max}}}^{n-|\tilde{\beta}|}\left(T'_{j_{\textrm{max}}}\right)^{\tilde{\beta}}$ with 
$|\tilde{\beta}| \leq m-1$.

If instead $\beta_{j_{\textrm{max}}} \leq n-m$, then $|\beta'_i| + n - |\beta| \geq m$ for any entry 
$\beta_i$ of $\beta$. \footnote{Note that we must have $d-1 \geq 2$ in this case so that the expression $|\beta'_i|$ 
makes sense.} Since, by assumption $|\beta| \geq m$, this means $T^{\alpha}=T_j^{n-|\beta|}(T'_j)^{\beta} = 0$, by 
Lemma \ref{T_i^mT_j^m=0}.
\end{proof}

We are now able to answer our initial question for the case $m \in \N$ and $\mu=1$. That is, we can determine the 
intersection class of $(m,p)$- and $(1,\infty)$-isometric tuples on a given space $X$. 

\begin{theorem}\label{(m,p) and (1,inf)}
Let $T =(T_1,...,T_d) \in B(X)^d$ be an $(m,p)$-isometric tuple and also a $(1, \infty)$-isometric
tuple. Then one operator $T_{j_0}$ is an isometry and all other operators 
satisfy $(T_{j_0}')^{\beta}=0$ for all $\beta \in \N^{d-1}$ with $|\beta|=m$, and are, in particular, nilpotent of order 
$\leq m$.
\end{theorem}
\begin{proof}
Again, we can assume that $m \geq 1$.

Since $T$ is an $(m,p)$-isometric tuple, by Proposition \ref{Newtonform}.(i),
\begin{align}\label{recursion in (m,p) and (1,inf)}
	Q^{n}(x)=\sum_{k=0}^{m-1}n^{(k)}\left(\frac{1}{k!}P_k(x)\right), \ \ \forall x \in X, \ \forall n \in \N.
\end{align}
Where $n^{(k)}=\binom{n}{k}k!=n (n-1) ... (n-k+1)$. That is, for all $x \in X$,
the sequence $\left(Q^n(x)\right)_{n \in \N}$ is interpolated by a polynomial of 
degree of less or equal to $m-1$.

Now, by Corollary \ref{Cor T_i^mT_j^m=0} above, for $n \geq 2m-1$, $n \in \N$, $Q^n(x)$ reduces to
\begin{align*} 
	Q^n(x)= \sum_{\substack{\beta \in \N^{d-1} \\ |\beta|=0,...,m-1}} \sum_{j=1}^d 
	\frac{n!}{(n-|\beta|)!\beta!} \norm{T_j^{n-|\beta|}(T'_j)^{\beta}x}^p,
 	\ \ \forall x \in X,
\end{align*}
where $\frac{n!}{(n-|\beta|)!\beta!}= \frac{n^{(|\beta|)}}{\beta!}$. (We set $n \geq 2m-1$, so that we don't get any multi-indices twice in this expression.) 
 
Further, for each $n \in \N$, $k \in \{0,...,m-1\}$, $\beta \in \N^{d-1}$, $j \in \{1,...,d\}$ and all $x \in X$, by 
Proposition \ref{existence of j_x for (1,inf)}, there exists an $\ell_j \in \{1,...,d\}$, such that
\begin{align*}
\norm{T_{\ell_j}^{\nu}\left(T_j^{n-k}(T'_j)^{\beta}x\right)}
= \norm{T_j^{n-k}(T'_j)^{\beta}x}, \ \  \forall \ \nu \in \N.
\end{align*}
By Corollary \ref{Cor T_i^mT_j^m=0}, for $n \geq 2m-1$, $n \in \N$, we must have $\ell_j = j$, i.e. 
\begin{align*}
\norm{T_j^{\nu+n-k}(T'_j)^{\beta}x} 
= \norm{T_j^{n-k}(T'_j)^{\beta}x}, \ \ \forall \ n, \nu \in \N, \ n \geq 2m-1.
\end{align*}
But that means that, for all $k \in \{0,...,m-1\}$, $\beta \in \N^{d-1}$, $j \in \{1,...,d\}$ and all $x \in X$, the 
sequences $\left(\norm{T_j^{n-k}(T'_j)^{\beta}x}^p\right)_{n \in \N}$ becomes constant for $n \geq 2m-1$.

Therefore, for all $x \in X$, for $n \geq 2m-1$, the sequence $\left(Q^n(x)\right)_{n \in \N}$ is
interpolated by the polynomial 
\begin{align*} 
	n \mapsto 
	\sum_{\substack{\beta \in \N^{d-1} \\ |\beta|=0,...,m-1}} \sum_{j=1}^d 
	\frac{n^{(|\beta|)}}{\beta!}\norm{T_j^{2m-1-|\beta|}(T'_j)^{\beta}x}^p,
\end{align*}
which is of degree less or equal to $m-1$.

However, this polynomial must be the same as the one in \eqref{recursion in (m,p) and (1,inf)}. In particular, their 
coefficients have to be equal and, more particularly, equating constants, we must have 
\begin{align*}
\sum_{j=1}^d \norm{T_j^{2m-1}x}^p = \norm{x}^p, \ \ \forall x \in X.
\end{align*}
Take now $j_0 \in \{1,...,d\}$ and $x_{j_0} \in X_{j_0}$, for $X_{j_0}$ defined as in Remark 
\ref{remark if T (m,infty)-tuple, X is union}. 
Then $\sum_{\substack{j=1 \\ j \neq j_0}}^d \norm{T_j^{2m-1}x_{j_0}}^p = 0$ and, thus, 
$x_{j_0} \in \ker T_j^{2m-1}$ for all $j \in \{1,...,d\}$ with $j \neq j_0$.

Since $x_{j_0} \in X_{j_0}$ and $j_0 \in \{1,...,d\}$ were chosen arbitrarily, $X_j \in \ker T_i^{2m-1}$ for all 
$i \neq j$. Then it follows from Lemma \ref{when (1,inf) implies isometry} that one of the 
operators $T_1,...,T_d$ is an isometry.

Let $T_{j_0}$ be isometric. Then we have $(T_{j_0}')^{\beta}=0$ for all $\beta \in \N^{d-1}$ with $|\beta|=m$ by 
Proposition \ref{isometry with nilpotent}.
\end{proof}

The case $m=1$ and $\mu \in \N$, $\mu \geq 1$, now follows easily.\footnote{It is actually easy to find an elementary 
proof for Corollay \ref{(1,p) and (m,infty)}. However, it is 
more elegant to deduce the statement from Proposition \ref{(m,p) and (1,inf)}.}

\begin{corollary}\label{(1,p) and (m,infty)} 
Let $T=(T_1,...,T_d) \in B(X)^d$ be a $(1,p)$-isometric tuple and also a $(\mu,\infty)$-isometric tuple. Then one 
operator $T_{j_0}$ is an isometry and $T=(0,...,0,T_{j_0},0,...,0)$.
\end{corollary}
\begin{proof}
Since $T$ is a $(1,p)$-isometric tuple, Proposition \ref{Newtonform}.(i) gives,\\
$\sum_{|\alpha|=n}\frac{n!}{\alpha!}\norm{T^{\alpha}x}^p=\norm{x}^p$, for all $n \in \N$ and all $x \in X$.
Consequently, $\norm{x} \geq \norm{T^{\alpha}x}$, for any multi-index $\alpha \in \N^d$, for all $x \in X$. I.e.
\begin{align*}
	\max_{\alpha \in \N^d}\norm{T^{\alpha}x} = \norm{x}, \ \forall x \in X.
\end{align*}
Then, since $T$ is a $(\mu,\infty)$-isometric tuple, $T$ is already a
$(1,\infty)$-isometric tuple, by Theorem \ref{equivalent norm (m,inf)}. The result now follows from the preceding statement and Proposition 
\ref{isometry with nilpotent}.
\end{proof}

We can now prove the case where our tuple is constructed, using an $(m,p)$-isometric operator (i.e., by applying 
Proposition \ref{proposition example (m,p)-tuples made of (m,p)-operators}).

\begin{proposition}\label{(m,p) and (mu,inf) made of m-isometries}
Let $T=(T_1,...,T_d) \in B(X)^d$ be an $(m,p)$-isometric tuple of the form of 
Proposition \ref{proposition example (m,p)-tuples made of (m,p)-operators}. That is, $T_j=z_jS$, where $S \in B(X)$ is 
an $(m,p)$-isometric operator and $z:=(z_1,...,z_d) \in \K^d$ with $\norm{z}_{p}=1$. Assume further that $T$ is 
additionally a $(\mu,\infty)$-isometric tuple. Then 
$T=(0,...,0,z_{j_0}S,0,...,0)$ with $|z_{j_0}|=1$ and $S$ being an isometry. In particular, $T_{j_0}=z_{j_0}S$ is an isometry.  
\end{proposition}
\begin{proof}
Since $T$ is an $(\mu,\infty)$-isometry, for all $x \in X$, the family\\
$\left(\norm{T^{\alpha}x}\right)_{\alpha \in \N^d}=\left(|z^{\alpha}|\norm{S^{|\alpha|}x} \right)_{\alpha \in \N^d}$ 
attains its maximum. Since $\pnorm{z}{p}=1$ by assumption, we have $\max_{|\alpha|=n}|z^{\alpha}|=1$ for all $n \in \N$. 
This forces $(\norm{S^nx})_{n \in \N}$ to be bounded for every 
$x \in X$. Then, since $S$ is an $(m,p)$-isometric operator, $S$ is an isometry by
\cite[Proposition 2.1]{HoMaOs}. Hence, by Proposition 
\ref{proposition example (m,p)-tuples made of (m,p)-operators}, $T$ is a $(1, p)$-isometric tuple. Then Corollary 
\ref{(1,p) and (m,infty)} forces one operator $z_{j_0}S$ to be an isometry.

Now Proposition \ref{isometry with nilpotent} gives $(z_{i}S)^m=0$ for all $i \in \{1,...,d\}$ with $i \neq j_0$. Since 
$S$ is an isometry, this is only possible if $z_i=0$ for all $i \in \{1,...,d\}$ with $i \neq j_0$. Thus, 
$T=(0,...,0,z_{j_0}S,0,...,0)$ with $|z_{j_0}|=1$.
\end{proof}

More general results appear difficult to obtain at the moment. We present in the remaining parts some partial results in 
the case $m=\mu=2$.

The next lemma simply states that one cannot increase a maximum.

\begin{lemma}\label{for max index norms are equal}
Let $T=(T_1,...,T_d) \in B(X)^d$ be a $(\mu,\infty)$-isometric tuple. For each $x \in X$ and each
$\tilde{\alpha}(x) \in \N^d$ with $\max_{\alpha \in \N^d}\norm{T^{\alpha}x}=\norm{T^{\tilde{\alpha}(x)}x}$, we have 
$\norm{T^{\tilde{\alpha}(x)}x}=|T^{\tilde{\alpha}(x)}x|_{\infty}$.
\end{lemma}
\begin{proof}
Fix $x \in X$ and let $\tilde{\alpha}(x) \in \N^d$, such that 
$\max_{\alpha \in \N^d}\norm{T^{\alpha}x}=\norm{T^{\tilde{\alpha}(x)}x}$. We have
\begin{align*}
	\norm{T^{\tilde{\alpha}(x)}x} \leq |T^{\tilde{\alpha}(x)}x|_{\infty}
	= \max_{\alpha \in \N^d}\norm{T^{\alpha}T^{\tilde{\alpha}(x)}x}
	\leq  \max_{\alpha \in \N^d}\norm{T^{\alpha}x} = \norm{T^{\tilde{\alpha}(x)}x}.
\end{align*}
\end{proof}

\begin{proposition}\label{prop (2,p) and (2,inf)}
Let $T=(T_1,...,T_d) \in B(X)^d$ be a $(2,p)$-isometric tuple and a $(2, \infty)$-isometric tuple.
\begin{itemize}
\item[(i)] The sequences $\left(\norm{T^n_jx}\right)_{\substack{n \in \N \\ n \geq 2}}$ are constant for all $x \in X$, 
for all $j \in \{1,...,d\}$.
\item[(ii)] For all $n \geq 2$, $n \in \N$,
\begin{align*}
	\big \|\sum_{j=1}^d T_j^n x \big \|^p = \sum_{j=1}^d \norm{T_j^nx}^p = \norm{x}^p, \ \ \forall x \in X.
\end{align*}
In particular, the tuple $T^2:=(T_1^2,...,T_d^2)$ is an $\ell_p$-spherical isometry and the operator 
$\sum_{j=1}^d T_j^2$ is an isometry.
\item[(iii)] We have $T^2_{j_0}=0$ for some $j_0 \in \{1,...,d\}$.
\end{itemize}
\end{proposition}
\begin{proof}
We proof (i) and (ii) together.

(i) + (ii): Since $T$ is a $(2,p)$-isometric tuple, by Proposition \ref{Newtonform}.(i),
\begin{align}\label{recursion in (2,p) and (2,inf)}
	Q^{n}(x)=n P_1(x)+ \norm{x}^p, \ \ \forall x \in X, \ \forall n \in \N.
\end{align}
That is, for all $x \in X$, the sequence
$\left(Q_n(x)\right)_{n \in \N}$ is interpolated by a polynomial of 
degree less or equal to $1$.

Now, by Corollary \ref{Cor T_i^mT_j^m=0}, for $n \geq 2$, $n \in \N$, $Q^n(x)$ reduces to
\begin{align}\label{Q^n for (2,p) and (2,inf)}
	Q^n(x)= n \left(\sum_{i=1}^{d}\sum_{\substack{j=1 \\ j \neq i}}^{d}\norm{T_i^{n-1}T_jx}^p \right)
	+ \sum_{j=1}^d\norm{T_j^nx}^p,
 	\ \ \forall x \in X.
\end{align}
Since $T$ is a $(2,\infty)$-isometry, for all $x \in X$, for all $i,j \in \{1,...,d\}$,\\
$\max_{\alpha \in \N^d}\norm{T^{\alpha}T_i^2T_jx}=\max_{|\alpha|=1}\norm{T^{\alpha}T_i^2T_jx}$, by Corollary 
\ref{max (m,inf)}. If $i \neq j$, by Corollary \ref{Cor T_i^mT_j^m=0}, we deduce that 
$\max_{|\alpha|=1}\norm{T^{\alpha}T_i^2T_jx}=\norm{T_i^3T_jx}$, for all $x \in X$.

However, then Lemma \ref{for max index norms are equal} gives that $\norm{T_i^3T_jx}=|T_i^3T_jx|_{\infty}$, 
for all $x \in X$. But then, for all $x \in X$,
\begin{align*}
	\norm{T_i^3T_jx}=|T_i^3T_jx|_{\infty}=\max_{|\alpha|=1}\norm{T^{\alpha}T_i^3T_jx}
	=\norm{T_i^4T_jx}=|T_i^4T_jx|_{\infty}, 
\end{align*}
for all $i \neq j$.

By repeating this process ad infinitum, we have for all $i \neq j$,
\begin{align*} 
	\norm{T_i^3T_jx}=\norm{T_i^4T_jx}=|T_i^4T_jx|_{\infty}
	=\norm{T_i^5T_jx}=|T_i^5T_jx|_{\infty}=...=\norm{T_i^nT_jx}, 
\end{align*}
for all $n \geq 3$, $n \in \N$, for all $x \in X$.

Therefore, the sequences $\left(\norm{T_i^{n-1}T_jx}\right)_{\substack{n \in \N \\ n \geq 4}}$ are constant, for all 
$i \neq j$, for all $x \in X$.

By equating \eqref{recursion in (2,p) and (2,inf)} and \eqref{Q^n for (2,p) and (2,inf)}, we get for all $n \geq 2, n \in \N$
\begin{align*}
	\sum_{j=1}^d\norm{T_j^nx}^p 
	= n \left(P_1(x) - \sum_{i=1}^{d}\sum_{\substack{j=1 \\ j \neq i}}^{d}\norm{T_i^{n-1}T_jx}^p \right)+ \norm{x}^p, 
	\ \ \forall x \in X.
\end{align*}
The left hand side is non-negative and bounded, for all $x \in X$, by Corollary \ref{max (m,inf)}, thus, so has to be the right hand side. Since
$ \sum_{i=1}^{d}\sum_{\substack{j=1 \\ j \neq i}}^{d}\norm{T_i^{n-1}T_jx}^p 
= \sum_{i=1}^{d}\sum_{\substack{j=1 \\ j \neq i}}^{d}\norm{T_i^{3}T_jx}^p$ 
is constant for $n \geq 4$, this forces
\begin{align}\label{equation P_1 id m=mu=d=2} 
P_1(x) = \sum_{i=1}^{d}\sum_{\substack{j=1 \\ j \neq i}}^{d}\norm{T_i^{3}T_jx}^p, \ \ \forall x \in X. 
\end{align}
Therefore, 
\begin{align}\label{equation T^4 is (1,p)}
	\sum_{j=1}^d\norm{T_j^nx}^p = \norm{x}^p, \ \ \forall n \geq 4, n \in \N, \forall x \in X.
\end{align}
Since $T_i^2T_j^2 = 0$ for all $i \neq j$ by Lemma \ref{T_i^mT_j^m=0}, replacing $x$ by $T^{\nu}_{j_0}x$ 
for $\nu \in \N$ with $\nu \geq 2$ in this last equation 
gives $\norm{T^{\nu}_{j_0}x} = \norm{T_{j_0}^{n+\nu}x}$ for all $n \geq 4$, $n \in \N$, for all $x \in X$. 

Hence, the sequences $\left(\norm{T_j^{k}x}\right)_{\substack{n \in \N \\ n \geq 2}}$ are constant for all 
$j \in \{1,...,d\}$, for all $x \in X$. This is (i).

Combining now (i) and \eqref{equation T^4 is (1,p)}, gives that we actually have
\begin{align*}
	\sum_{j=1}^d\norm{T_j^nx}^p = \norm{x}^p, \ \ \forall n \geq 2, n \in \N, \forall x \in X,
\end{align*}
which is one of the two equations we had to show for (ii). Now replace $x$ by $\sum_{i=1}^dT_i^{\nu}x$ for 
$\nu \in \N$ with $\nu \geq 2$ in this last equation. Then, again, since $T_i^2T_j^2 = 0$ for all $i \neq j$, we get 
that for all $n, \nu \geq 2$, $n, \nu \in \N$,
\begin{align*}
	\big \|\sum_{j=1}^d T_j^{\nu} x \big \|^p = \sum_{j=1}^d \norm{T_j^{n+\nu}x}^p 
	= \sum_{j=1}^d \norm{T_j^{n}x}^p, \ \ \forall x \in X. 
\end{align*}
This is the second equation we had to show.

(iii): The equation in (ii) implies that $\norm{x} \geq \max_{j=1,...,d} \norm{T_j^3x}$, for all $x \in X$. 
\begin{itemize}
\item[(iii.a)] If $\norm{x}=\norm{T_j^3x}$, for some $j \in \{1,...,d\}$ then obviously 
$\norm{T_i^3x}=0$ and $x \in N(T_i^3)$ for all $i \neq j$, again by (ii).
\item[(iii.b)] Assume $\norm{x} > \max_{j=1,...,d} \norm{T_j^3x}$. (Note that this implies $d>1$ and that we have more 
than one non-zero operator.)

Since $|x|_{\infty} \geq \norm{x}$ and 
$|x|_{\infty}=\max_{\alpha \in \N^d}\norm{T^{\alpha}x}=\max_{|\alpha|=3}\norm{T^{\alpha}x}$, it follows that
\begin{align*}
	|x|_{\infty}= \max_{\substack{i=1,...,d \\ j=1,...,d \\ j\neq i}}\norm{T_i^2T_jx}.
\end{align*}
(Since $T_iT_jT_k=0$ for distinct $i,j,k$ by Lemma \ref{T_i^mT_j^m=0}.)

Let $|x|_{\infty}=\norm{T_{i_0}^2T_{j_0}x}$ for some $i_0 \neq j_0$, $i_0,j_0 \in \{1,...,d\}$. Then
\begin{align*}
	&\norm{T_{j_0}x}^p= \sum_{\substack{j=1,...,d \\ j \neq i_0, j_0}} \norm{T_j^2T_{j_0}x}^p 
	+ \norm{T_{i_0}^2T_{j_0}x}^p + \norm{T^3_{j_0}x}^p \\
	\Leftrightarrow \ \ 
	&\norm{T_{j_0}x}^p= \sum_{\substack{j=1,...,d \\ j \neq i_0, j_0}} \norm{T_j^2T_{j_0}x}^p 
	+ |x|_{\infty}^p + \norm{T^3_{j_0}x}^p.
\end{align*}
Since $|x|_{\infty} \geq \norm{T^{\alpha}x}$ for all $\alpha \in \N^d$ simply by definition, we certainly have 
$x \in N(T_{j_0}^3)$.
\end{itemize}

We conclude that in any case, $x \in N(T_j^3)$ for some $j \in \{1,...,d\}$. 

In other words, 
$X= \bigcup_{j=1,...,d} N(T_j^3)$. Hence, $N(T_{j_0}^3)= X$, i.e. $T_{j_0}^3=0$ for one $j_0 \in \{1,...,d\}$. This 
gives $T_{j_0}^2=0$ by (i).
\end{proof}

\begin{corollary}
Let $T=(T_1,...,T_d) \in B(X)^d$ be a $(2,p)$-isometric tuple and a $(2, \infty)$-isometric tuple. Assume further that 
one of the following holds:
\begin{itemize}
\item[(i)] One of the operators $T_1,...,T_d$ is injective.
\item[(ii)] One the operators $T_1,...,T_d$ is surjective.
\item[(iii)] $T$ does not contain a non-zero nilpotent operator. 
\item[(iv)] We have $d=2$.
\item[(v)] We have $p=1$ and $X$ is stricly convex.
\item[(vi)] We have $0< p < 1$.
\end{itemize}
Then one operator $T_{j_0}$ is an isometry and all other operators 
satisfy $(T_{j_0}')^{\beta}=0$ for all $\beta \in \N^{d-1}$ with $|\beta|=2$, and are, in particular, nilpotent of order 
$\leq 2$. In case (iii), $T$ consists actually of one isometry and zeros. 
In case (ii), $T_{j_0}$ is actually an isometric isomorphism. 
\end{corollary}
\begin{proof}
(i): Let $T_{j_0}$ be injective. Then, by Lemma \ref{T_i^mT_j^m=0}, all the other operators $T_j$ for $j \neq j_0$ 
are nilpotent of degree $\leq 2$. In particular, by Proposition \ref{prop (2,p) and (2,inf)}.(ii), $T^2_{j_0}$ is an 
isometry. Without loss of generality, assume $T_1^{2}$ is an isometry and $T_j^{2}=0$, for all $j \in \{2,...,d\}$. 
Then, by \eqref{equation P_1 id m=mu=d=2} and since $\left(\norm{T_j^nx}\right)_{\substack{n \in \N \\ n \geq 2}}$ are 
constant,
\begin{align*}
	P_1(x)= \sum_{i=1}^{d}\sum_{\substack{j=1 \\ j \neq i}}^{d}\norm{T_i^{2}T_jx}^p = \sum_{j=2}^{d}\norm{T_jx}^p.
\end{align*}
Now, by definition, $P_1(x)=-\norm{x}^p+\sum_{j=1}^d\norm{T_jx}^p$, for all $x \in X$. Therefore, 
\begin{align*}
	-\norm{x}^p+\sum_{j=1}^d\norm{T_jx}^p = \sum_{j=2}^{d}\norm{T_jx}^p
	 \ \ \Leftrightarrow \ \ -\norm{x}^p+\norm{T_1x}^p=0, \ \ \forall x \in X,
\end{align*}
and $T_1$ is an isometry. Then we have $(T_{1}')^{\beta}=0$ for all $\beta \in \N^{d-1}$ with $|\beta|=2$ by 
Proposition \ref{isometry with nilpotent}

(ii): Let $T_{j_0}$ be surjective. Then $T^2_{j_0}$ is surjective and, since $\norm{T^3_{j_0}x}=\norm{T^2_{j_0}x}$ for 
all $x \in X$, by Proposition \ref{prop (2,p) and (2,inf)}.(i), the operator $T_{j_0}$ is actually an isometry. Then 
$(T_{j_0}')^{\beta}=0$ for all $\beta \in \N^{d-1}$ with $|\beta|=2$ by Proposition \ref{isometry with nilpotent}

(iii): Assume $T=(T_1,...,T_d)$ does not contain a non-zero nilpotent operator. Then, by Proposition 
\ref{prop (2,p) and (2,inf)}.(iii), 
$T$ must contain an operator which is the zero-operator, say $T_d$. But then we can reduce $T$ to $(T_1,...,T_{d-1})$ 
and repeat the argument until $T=(T_{j_0})$ for some isometric operator $T_{j_0}$.  

(iv): Let $T=(T_1, T_2)$. By Propposition \ref{prop (2,p) and (2,inf)}.(iii), one operator $T_2$, say, is nilpotent of 
degree $\leq 2$. Then, \ref{prop (2,p) and (2,inf)}.(ii) forces $T_1^2$ to be an isometry. But then $T_1$ is injective 
and the result follows from (i).

(v): If $p=1$, by Proposition \ref{prop (2,p) and (2,inf)}.(ii) we have that, for each $x \in X$, the triangle 
inequality becomes an equality for the vectors $T_j^2x$. If $X$ is strictly convex, this implies that there exist
$\lambda_{i,j, x} \in \R$ with $T^2_ix = \lambda_{i,j, x}T^2_jx$ for all $i, j$, for all $x \in X$. But then 
$T_j^2T_i^2 = 0$ for all $i \neq j$ implies $\lambda_{i,j, x}=0$ or $T_i^2x=T_j^2x = 0$. Therefore, there exists a 
$j_0$ with $T^2_j=0$ for all $j \neq j_0$ and, by Proposition \ref{prop (2,p) and (2,inf)}.(ii), $T^2_{j_0}$ is an isometry. 
Thus, $T_{j_0}$ is injective and the statement follows from (i).

(vi):Assume $0< p < 1$. Then, by Proposition \ref{prop (2,p) and (2,inf)}.(ii),
\begin{align*}
	\norm{x}=\left(	\sum_{j=1}^d \norm{T_j^2x}^p \right)^{1/p} &\geq \sum_{j=1}^d \norm{T_j^2x} 
	\geq \norm{ \sum_{j=1}^d T_j^2x} = 
	\norm{x}, \ \ \forall x \in X.
\end{align*}
That is, $\norm{x}=\sum_{j=1}^d \norm{T_j^2x} = \left(\sum_{j=1}^d \norm{T_j^2x}^p \right)^{1/p}$ for all $x \in X$. (In 
particular, $(T_1^2,...,T_d^2)$ is a $(1,p)$-isometry and a $(1,1)$-isometry.)
Hence, for all $x \in X$, the vectors $(\norm{T_1^2x} ,...,\norm{T_d^2x})$ lie on the 
same circle with radius $\norm{x}$ in $\R^d$ with respect to both $\pnorm{.}{1}$ and $\pnorm{.}{p}$. This means that, 
for all $x \in X$, $(\norm{T_1^2x} ,...,\norm{T_d^2x})$ is of the form $(\norm{x},0,...,0)$ with respect to some 
permutation (which, at this stage of the argument, may depend on $x$). Hence, each $x \in X$ is in the kernel of $d-1$ 
operators $T_1^2,...,T_d^2$ while the remaining operator, say $T^2_{j_0}$, acts isometrically on $x$. But then
$\max_{j=1,...,d} \norm{T_j^2x}=\norm{x}$ for all $x \in X$ and $(T_1^2,...,T_d^2)$ is a $(1,\infty)$-isometry. By Proposition \ref{(1,p) and (m,infty)}, $T^2_{j_0}$ is an isometry and therefore $T_{j_0}$ is injective.
\end{proof}

These results lead to the following question:

\begin{question}
Is every tuple of commuting bounded linear operators\\ $T = (T_1,...,T_d) \in B(X)^d$ which is simultaneously 
an $(m,p)$-isometric and an $(\mu,\infty)$-isometric tuple actually of the form of Proposition 
\ref{isometry with nilpotent}? 
\end{question}

That is, is one operator $T_{j_0}$ an isometry and all other operators 
satisfy $(T_{j_0}')^{\beta}=0$ for all $\beta \in \N^{d-1}$ with $|\beta|=m$, and are, in particular, nilpotent of order 
$\leq m$? Moreover, is an $(m,p)$-isometric tuple an $(m,\infty)$-isometric tuple if, and only if, it is an 
$(m,q)$-isometric tuple for \emph{all} $q \in (0,\infty)$?\\

\section*{Acknowledgements}

The extension from $d=2$ to a general $d$-tuple in Proposition \ref{isometry with nilpotent} is due to Mícheál Ó 
Searcóid, who also improved the proof of Lemma \ref{lemma for D}. 

Lemma \ref{T_i^mT_j^m=0} was proved in a first version (for $\mu=1$ and $d=2$) by Hermann Render, who also motivated
Corollary \ref{asymptotic}.

The first author would like to thank the School of Mathematical Sciences, University College Dublin for funding this work through a Research Demonstratorship and a bursary.

\bigskip

\hrule
\begin{quote}
Philipp Hoffmann\\ 
National University of Ireland, Maynooth\\
Department of Computer Science\\
Maynooth, County Kildare\\
Ireland\\
\begin{tabular}{@{}l@{ }l}
\emph{email:} & \url{philipp.hoffmann@cs.nuim.ie} \\
 & \url{philip.hoffmann@maths.ucd.ie}
\end{tabular}
\end{quote}

\

\begin{quote}
Michael Mackey\\
University College Dublin\\
School of Mathematical Sciences\\
Belfield, Dublin 4\\
Ireland\\
\begin{tabular}{@{}l@{ }l}
\emph{email:} \url{mackey@maths.ucd.ie}
\end{tabular}
\end{quote}


\begin{thebibliography}{0}
\bibitem{aign} M. Aigner, Diskrete Mathematik, Vieweg, Braunschweig/Wiesbaden 1993.
\bibitem{ag} J. Agler, A disconjugacy theorem for Toeplitz operators, \emph{Am. J. Math.} 112(1) (1990), 1-14.
\bibitem{ag-st1} J. Agler and M. Stankus, $m$-isometric transformations of Hilbert space, I, \emph{Integr. equ. oper. 
theory} Vol. 21, No. 4 (1995), 383-429.
\bibitem{ag-st2} J. Agler and M. Stankus, $m$-isometric transformations of Hilbert space, II, 
\emph{Integr. equ. oper. theory} Vol. 23, No. 1 (1995), 1-48.
\bibitem{ag-st3} J. Agler and M. Stankus, $m$-isometric transformations of Hilbert space, III, 
\emph{Integr. equ. oper. theory} Vol. 24, No. 4 (1996), 379-421.
\bibitem{ba} F. Bayart, $m$-Isometries on Banach Spaces, \emph{Mathematische Nachrichten}, Volume 284, No. 17-18 (2011), 
2141-2147.
\bibitem{BeWa} M. A. Berger and Y. Wang, Bounded semigroups of matrices, \emph{Linear Algebra Appl.} 166 (1992), 21-27.
\bibitem{bemarmart} T. Bermúdez, I. Marrero and A. Martinón, On the Orbit of an \emph{m}-Isometry,
\emph{Integr. equ. oper. theory} Vol. 64, No. 4 (2009), 487-494.
\bibitem{bemane} T. Bermúdez, A. Martinón and E. Negrín, Weighted Shift Operators Which are m-Isometries,
\emph{Integr. equ. oper. theory} Vol. 68, No. 3 (2010), 301-312.
\bibitem{bemamu} T. Bermúdez, A. Martinón and V. Müller, $(m,q)$-isometries on metric spaces, \emph{J. Operator Theory} 72:2 (2014), 313-329.
\bibitem{bo} F. Botelho, On the existence of $n$-isometries on $\ell_p$ spaces, \emph{Acta Sci. Math. (Szeged)} 76 : 1-2 (2010), 183-192.
\bibitem{Cho-Zel} M. Ch\={o} and W. \.{Z}elazko, On geometric spectral radius of commuting $n$-tuples of operators, 
\emph{Hokkaido Math. J.} 21 (2) (1992), 251-258.
\bibitem{Ah-He} M. Faghih Ahmadi and K. Hedayatian, Hypercyclicity and Supercyclicity of $m$-Isometric Operators, \emph{Rocky Mountain J. Math.} Volume 42, Number 1 (2012), 15-23. 
\bibitem{gleari} J. Gleason and S. Richter, $m$-Isometric Commuting Tuples of Operators on a Hilbert Space,
\emph{Integr. equ. oper. theory} Vol. 56, No. 2 (2006), 181-196 .
\bibitem{Har} R.E. Harte, Spectral mapping theorems, \emph{Poc. Roy. Irish Acad. Sect.} A 72 (1972), 89-107.
\bibitem{HoMaOs} P. Hoffmann, M. Mackey and M. Ó Searcóid, On the second parameter of an $(m,p)$-isometry, \emph{Integr. 
equ. oper. theory} Vol. 71, No. 3 (2011), 389-405.
\bibitem{LiSo} C.-K. Li and W. So, Isometries of $\ell_p$-norm, \emph{Amer. Math. Monthly} 101 (1994), 452-453.
\bibitem{Mue} V. M\"{u}ller, On the joint spectral radius, \emph{Annales Polonici Mathematici} 66 (1997), 173-182.
\bibitem{MueSol} V. M\"{u}ller and A. Soltysiak, Spectral radius formula for commuting Hilbert space operators, \emph{Studia Math.} 103 (3) (1992), 329-333.
\bibitem{ri} S. Richter, Invariant subspaces of the Dirichlet shift, \emph{J. reine angew. Math.} 386 (1988), 205-220.
\bibitem{ri(talk)} S. Richter and C. Sundberg, A model for two-isometric operator tuples with finite defect, talk at the AMS Sectional
meeting, Cornell University, September 2011,
\emph{http://www.math.utk.edu/$\sim$richter/talk/Cornell\_Sept\_2011}.
\bibitem{RoStra} G.-C. Rota and W. G. Strang, A note on the joint spectral radius, \emph{Indag. Math.} 22 (1960), 379-381.
\bibitem{sa} O.A. Sid Ahmed, $m$-isometric Operators on Banach Spaces, \emph{Asian-Eur. J. Math.} 3, No. 1 (2010), 1-19.
\bibitem{Slo-Zel} Z. S\l odkowski and W. \.{Z}elazko, On joint spectra of commuting families of operators, \emph{Studia 
Math.} 50 (1974), 127-148.
\bibitem{Sol} A. Soltysiak, On the joint spectral radii of commuting Banach algebra elements, \emph{Studia Math.} 105 (1) 
(1993), 93-99.
\bibitem{Tay} J.L. Taylor, A Joint Spectrum for Several Commuting Operators, \emph{J. Functional Analysis} 6 (1970), 172-191.
\bibitem{Zel} W. \.{Z}elazko, On a problem concerning joint approximate point spectra, \emph{Studia Math.} 45 (1973), 239-240. 
\end{thebibliography}
\end{document}